%% file: thirdgeodsept.tex
\documentclass[12pt]{amsart}
\usepackage{geometry}                
\usepackage{graphicx}

\usepackage{amsmath}
\usepackage{amsfonts}
\usepackage{times}
\usepackage{enumerate}
\usepackage{amssymb}
\usepackage{hyperref}

\newtheorem{theorem}{Theorem}[section]

\newtheorem{lemma}[theorem]{Lemma}

\numberwithin{equation}{section}

\newtheorem{Thm}{Theorem}[section]

\newtheorem{Pro}[Thm]{Proposition}

\newenvironment{Pf}[1]
{\trivlist\item[]{\it #1\@. }}{\hspace*{\fill}$\Box$\endtrivlist}

\input{figsa}

\renewcommand{\epsilon}{\varepsilon}

\renewcommand{\rho}{\varrho}

\begin{document}














 
\title{Lengths of three simple periodic geodesics on a Riemannian $2$-sphere}

\author{ Yevgeny Liokumovich, Alexander Nabutovsky and Regina Rotman}
\date{August, 8, 2014}
\maketitle

\begin{abstract}
 Let $M$ be a Riemannian $2$-sphere. A classical theorem
of Lyusternik and Shnirelman asserts the existence of three distinct simple
non-trivial periodic geodesics on $M$. In this paper we prove that
there exist three simple periodic geodesics with lengths that
do not exceed $20d$,
where $d$ is the diameter of $M$.
\par
We also present
an upper bound that depends only on the area and diameter for the lengths
of the three simple periodic geodesics with positive indices that
appear as minimax critical values in the classical proofs of
the Lyusternik-Shnirelman theorem.
\par
Finally, we present
better bounds for these three lengths for ``thin" spheres, when the area $A$ is much
less than $d^2$, where the bounds for the lengths of the first two simple periodic geodesics 
are asymptotically optimal, when ${A\over d^2}\longrightarrow 0$.
\end{abstract}

\section{Main results.} \label{Main}

\begin{Thm} \label{Theorem1.1} 
 Let $M$ be a Riemannian manifold of diameter $d$ diffeomorphic
to $S^2$. There exist three distinct non-trivial simple periodic geodesics
of length not exceeding $20d$ on $M$.
\end{Thm}
\medskip

Recall that  the term
``simple" means that the geodesics do not have self-intersections.
\par
The upper bounds $5d$ and $10d$ for lengths of two shortest
simple geodesics were found in [NR], but there we were not able to estimate
the length of the third simple periodic geodesic.
\par
Note that one can get a better upper
bound $4d$ for the length of the shortest non-trivial but not necessarily
simple periodic geodesic ([S], [NR0]). Yet the example of
a ``three-legged starfish" glued out of three 
long thin tentacles illustrates the difference between
majorization of the length of the shortest periodic geodesic and the length
of the shortest simple periodic geodesic: The shortest periodic geodesic
will be a figure eight shaped self-intersecting
curve hugging two of three tentacles. On the other hand it seems that
the shortest simple periodic geodesic will be
a curve that goes all the way along one of the tentacles and passes
through its top and the saddle point where two other tentacles intersect.
This example suggests that the approach developed in [NR] and [Sa] is not
applicable for majorization of lengths of simple periodic geodesics.
It also suggests that apparently there is no upper bound for the length
of the shortest simple periodic geodesic in terms of the area of $M$ alone.
On the other hand the examples of ellipsoids that are very close
to the standard sphere demonstrate that 1) there can be only
three simple periodic geodesics; and that 2) the length of the fourth periodic geodesic
(simple or not) cannot be majorized in terms of the diameter, the area and
even an upper bound for the absolute value of the curvature of the sphere.
(This is classical result of M. Morse, who proved that
the fourth periodic geodesic becomes uncontrollably large
for ellipsoids with distinct but very close semiaxes. Also,
a very nice extension of this result of Morse  for Riemannian
$2$-spheres of almost constant curvature was
proven by W. Ballmann in [B1].)
So, in this respect our result is the optimal possible. Finally note
that techniques from [LNR] lead to upper bound for the lengths
of three simple periodic geodesics in terms of the diameter {\it and}
the area of $M$, that are, however, much better than the estimates in [NR]
and the present paper, in the case
when the area is much smaller than the square of the
diameter (see Theorem 1.3 below). 
\par
The main idea of the proof of Theorem 1.1 is to express three homology classes of the space
of non-parametrized curves that are used in classical proofs of the Lyusternik-Shnirelman theorem by cycles that consist of simple closed curves ``mainly made''
of curves in a meridian-like family that connects two fixed points of $M$.
Following ideas of [Cr] and [NR] and using some observation of [M]
we attempt to construct such a family where the lengths of all meridians
are bounded by $const\ d$ for an appropriate $const$. Our repeated attempts
can be blocked only by appearence of different ``short" simple periodic
geodesics of index $0$. So, we either
get three short simple periodic geodesics of index
$0$, or our third attempt to construct
a ``meridional slicing" succeeds. Once one of our attempts succeeds,
and we get a slicing of $M$ into short meridians,
the original
proof of the Lyusternik-Shnirelman theorem yields the desired upper bounds.
\par
Also note that the standard proofs of the Lyusternik-Shnirelman theorem produce
three simple geodesics that have positive indices (cf. [T]). 
In fact, these geodesics appear as minimax critival values
corresponding to certain families of $1$-, $2$-, and $3$-dimensional
cycles with $\mathbb{Z}_2$ coefficients in the space of 
non-parametrized simple curves on the Riemannian $2$-sphere that
correspond to the gernerators of the respective homology groups of this space.
However,
our quantitative version of Lyusternik-Shnirelman produces three geodesics
that can be local minima of the length functional.
Can one find an effective version of the Lyusternik-Shnirelman theorem,
where one majorizes the lengths of three simple geodesics with positive
indices that appear in the original proof and are minimaxes for certain
specific families of simple closed curves? We believe that there are no such estimates solely in terms of
the diameter of $M$, and that Riemannian
metrics constructed in [L] using ideas from [FK]
provide a basis for counterexamples.
\par
On the other hand we will use the results from our paper [LNR] to prove 
that the lengths of three simple periodic geodesics that appear in the
classical proofs of the Lyusternik-Shnirelman theorem can be
majorized in terms of the volume and the area. In particular:
\medskip

\begin{Thm} \label{Theorem1.2}
Let $M$ be a Riemannian manifold diffeomorphic to the $2$-sphere of
area $A$ and diameter $d$. There
exist three distinct simple periodic geodesics on $M$ of length less than
$800d\max\{1,\ \ln {\sqrt{A}\over d}\}$ such that none of these
geodesics has index zero. (Moreover, these geodesics are exactly the
minimax geodesics that appear in the classical proofs of the Lyusternik-Shnirelman theorem).
\end{Thm}

\par
Finally, we consider the special case of ``thin" $2$-spheres. More
precisely, we provide 
upper bounds for the lengths of three simple periodic geodesics 
in terms of $d$ and the area, which are much better than the estimates
in terms of $d$, when ${A\over d^2}$ is very small.
As an extra bonus, our estimates are for the lengths of three geodesics
that appear in classical proofs of Lyusternik-Shnirelman and correspond to the generators
of the first three homology groups of the space of simple non-parametrized
curves.
\medskip

\begin{Thm} \label{Theorem1.3}
Let $M$ be a Riemannian manifold diffeomorphic to the $2$-sphere of
area $A$ and diameter $d$. There
exist three distinct simple periodic geodesics 
on $M$ 
such that none of these geodesics has index zero.
and with lengths that do not exceed, corresondingly, $d+700\sqrt{A}$,
$2d+1400\sqrt{A}$ and $4d+2800\sqrt{A}$.
\end{Thm}

We believe that the coefficient at $d$ in the first of these
estimates is optimal as evidenced by the example of a thin three-legged
starfish with tentacles of equal size. The example of a thin
long ellipsoid of revolution can be used to easily {\it prove} that $2$
is the optimal coefficient at $d$ in the second estimate. Also, one can try
to use the ideas of F. Balacheff, C. Croke and M. Katz from [BCK] , where
they constructed $2$-spheres where all non-trivial periodic geodesics were
longer than $2d$ by introducing a mild asymmetry to the Riemannian metric.
It seems very plausible that an application of the same idea to long thin
ellipsoids of revolution will result in Riemannian metrics on the
$2$-sphere, where the length of the second shortest simple periodic geodesic
is strictly greater than $2d$, thus, justifying the appearance of the $O(\sqrt{A})$
term. We believe that an $O(\sqrt{A})$ term is also required
in the first estimate. On the other hand, the coefficient $4$ in the estimate for the length
of the third shortest simple periodic geodesic on a ``thin" 
$2$-sphere is, most probably, not optimal.

\par

\section{Geometric realization of cycles from the proof of Lyusternik
and Shnireman.}

{\bf 2.1. A brief review of a classical proof of the existence of three
simple periodic geodesics}
The existence of three simple periodic
geodesics on every Riemannian $2$-sphere was first
proven by L.A. Lyusternik and L.G. Shnirelman ([LS], [Ly]).
They considered the space $\Pi M$ of non-parametrized
curves on a $2$-dimensional Riemannian manifold $M$ diffeomorphic
to $S^2$, as well as its subset $\Pi_0 M$ that consists of 
all constant curves (and, thus, can be identified with $M$).
Following [T] we consider the three relative homology classes
of the pair $(\Pi S^2,\Pi_0 S^2)$ with coefficients
in $\mathbb{Z}_2$, where we regard $S^2$ as the
unit round sphere in $\mathbb{R}^3$: 1) The $1$-dimensional class represented
by the relative $1$-cycle $z_1$ formed
by all circles on $S^2$ in planes parallel to the $XZ$-plane in the ambient $\mathbb{R}^3$ (see Fig. ~\ref{LS3});
2) The $2$-dimensional class represented by the relative cycle $z_2$ formed
by all circles in planes parallel to the $Z$-axis; and 3) The $3$-dimensional
class represented by the relative $3$-cycle $z_3$ formed by all round
circles on the sphere (including points regarded as ``degenerate" circles).
Lyusternik and Shnirelman described a curve shortening flow
in $\Pi M$.

\begin{figure}[center] 
\includegraphics[scale=0.3]{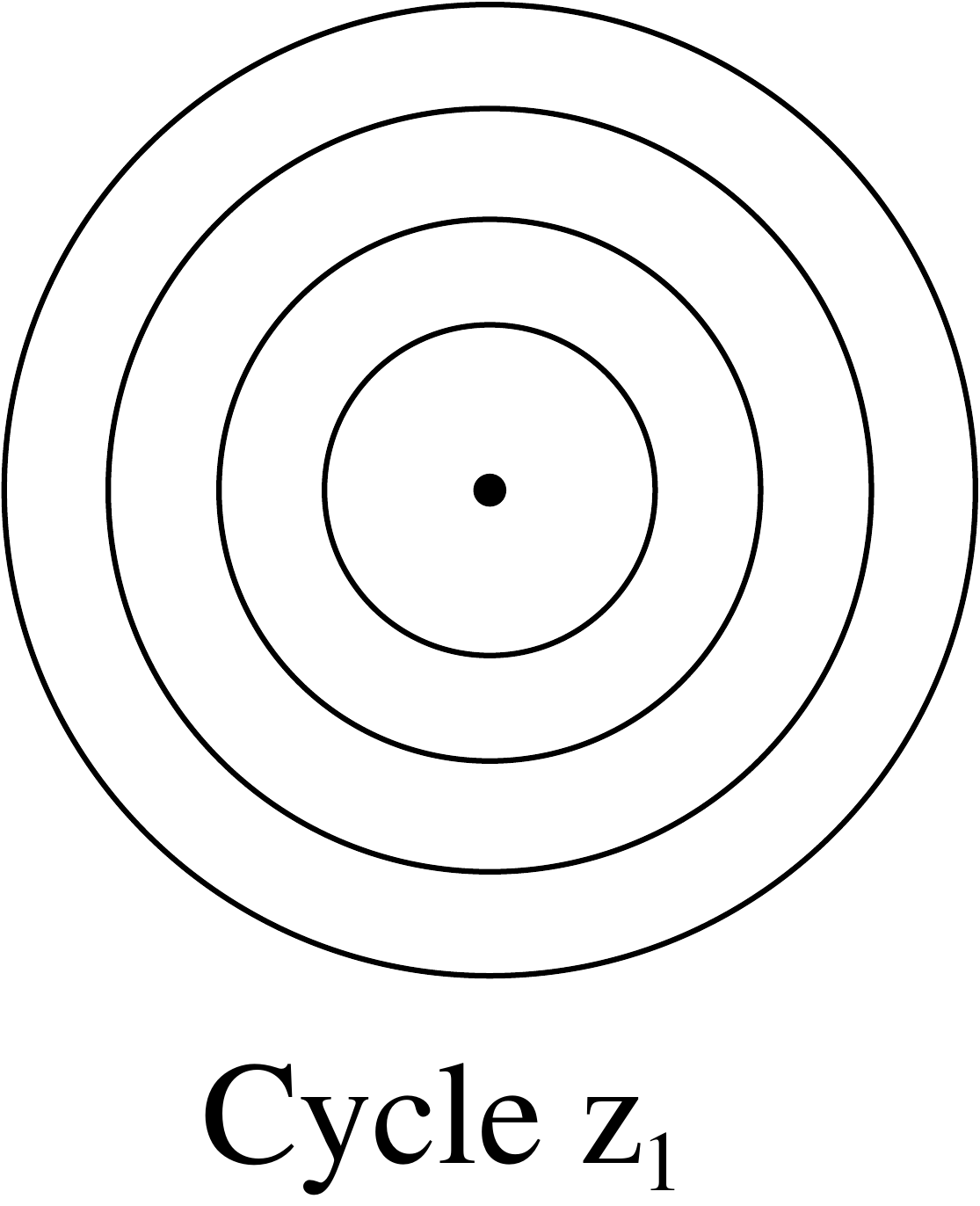} 
\caption{}
\label{LS3}
\end{figure}

They observed that this flow ``gets stuck"
on critical points representing simple periodic geodesics, when applied to $z_1$, $z_2$, $z_3$, and proved that if two of these three
geodesics coincide, then there is a whole critical level with a $1$-parametric
set of distinct simple periodic geodesics. Some errors in the construction 
of their curve shortening flow had been later corrected by W. Ballman
([B]), J. Jost ([J]) and I. Taimanov ([T0], [T]). Alternatively, one can 
prove the existence of three simple periodic geodesics using either
the curvature flow and Grayson's theorem ([Gr]) or an especially
simple curve shortening flow constructed by J. Hass and P. Scott ([HS]).
\par
An immediate corollary of any of those versions of
the proof of Lyusternik-Shnirelman
theorem is the following assertion:

\begin{Pro} \label {Proposition2.1}  Let $u_i$, $i=1,2,3$ be relative singular cycles of $(\Pi M, \Pi_0 M)$ with coefficients in $\mathbb{Z}_2$
such that $u_i$ is homologous to $z_i$ for each $i$. Assume that the
lengths of simple closed curves on $M$ in the image of each simplex of $u_i$
are bounded by $L_i$. 
Then there exist three distinct simple geodesics
on $M$ with lengths not exceeding $L_1,L_2$ and $L_3$, correspondingly.
Moreover, here we can also assume that these three simple geodesics
have positive
indices, i.e. none of them is a local minimum of the length functional on the space of simple closed curves on $M$.
\end{Pro}

Now assume that $f:S^2\longrightarrow M$ is a diffeomorphism
from the standard round sphere $S^2$ of radius $1$ in $\mathbb{R}^3$ onto $M$
that sends each meridian connecting the South pole $(0,0,-1)$ and the
North pole $(0,0,1)$ of $S^2$ into a curve of length $\leq L$.
Fix a (small) positive $\delta$.
 We would like to describe $1$-, $2$-, and $3$-dimensional relative
cycles $u_1$, $u_2$, $u_3$  in $(\Pi M,\ \Pi M_0)$
with coefficients in $\mathbb{Z}_2$ such that the lengths of all closed curves
on $M$ that appear as images of points of singular simplices of $u_i$
do not exceed, correspondingly, $2L+\delta$, $2L+\delta$ and $4L+\delta$. Moreover, if $L_0$
denotes the minimal length of the image of a meridian of $S^2$ under $f$,
then the lengths of closed curves in the images of simplices of $u_1$
will not exceed $L_0+L+\delta$. Here is the description of the
first two $u_1$ and $u_2$:
\par
{\bf 2.2. The cycles $u_1$ and $u_2$.} The cycles $u_i,\ i=1,2,3$ will be the images under the map induced by $f$
of the following cycles $v_1, v_2, v_3$ in $(\Pi S^2,\Pi S^2_0)$:
Assume that $m_0$ is a meridian such that the length of $f(m_0)$ is the minimal
length of the image under $f$ of a meridian of $S^2$. Then $v_1$
is constructed as follows. Denote
a meridian that forms an angle $\alpha$ with $m_0$ by $m_\alpha$.
Consider a $1$-dimensional family of simple curves formed by $m_\alpha$
and $m_{-\alpha}$, where $\alpha$ runs through the interval
$(\delta, \pi-\delta)$, where we will make a positive $\delta$ very small.
Now contract $m_\delta\bigcup m_{-\delta}$ through simple closed curves
that go along $m_{t\delta}$ and $m_{-t\delta}$ from the parallel with
lattitude $({\pi/2})t$ in the Southern hemisphere to the parallel with the
lattitude $({\pi/2})t$ in the Northern hemisphere and connects
between $m_{t\delta}$ and $m_{-t\delta}$ along these two parallels (see
Fig. ~\ref{LS5}).
Here $t$ varies between $0$ and $1$. When $t\longrightarrow 0$ these curves
converge to a point corresponding to the value $t=0$. Similarly, we can
contract $m_{-(\pi-\delta)}\bigcup m_{\pi-\delta}$. The resulting relative
$1$-cycle will be $v_1$.

\begin{figure}[center] 
\includegraphics[scale=0.3]{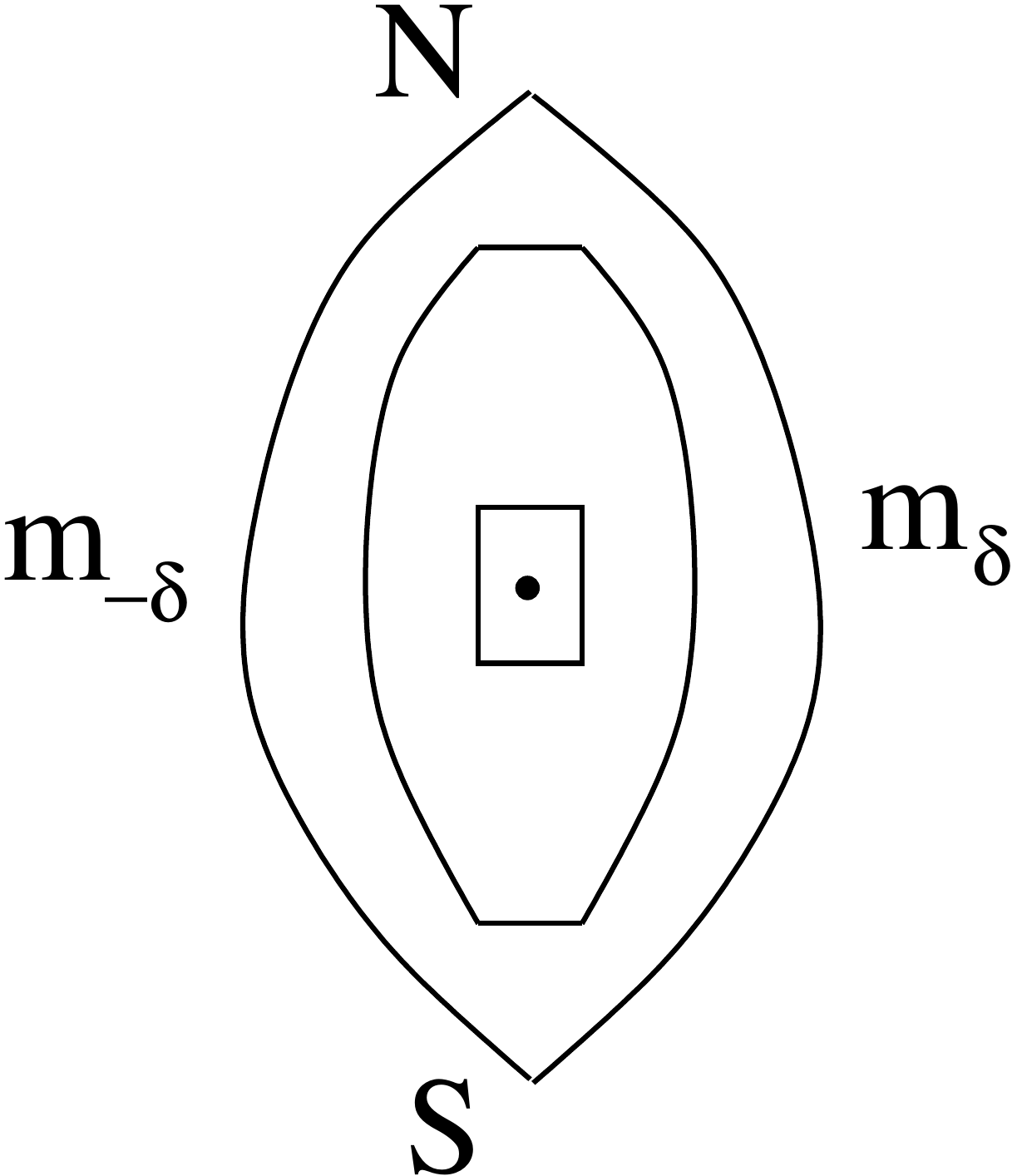} 
\caption{}
\label{LS5}
\end{figure}

\par
Note that we obtain a homologous cycle if instead of $m_\alpha\bigcup m_{-\alpha},\ \alpha\in [-\pi,\pi]$ we consider $m_0\bigcup m_\alpha$, $\alpha\in [0, 2\pi]$. If $m_0$
is the shortest meridian, than the maximal length of curves in these family can be less
than in the family $m_\alpha\bigcup m_{-\alpha}$.
\par
The 2-dimensional relative cycle $v_2$ is defined
by rotating $v_1$ by the angle $\pi$: For each $\theta\in [0,\pi]$
we take $m_\theta$ instead of $m_0$ and define a similar path in $\Pi S^2$ starting
and ending in $\Pi_0 S^2$. Denote the resulting relative
$1$-cycle by $v_{1t}$. Note that $m_\pi=m_{-\pi}$. Therefore,
$v_{1\pi}=v_{10}=v_1$. (Here it is important that the curves in $\Pi M$
are non-parametrized, and
the ring of coefficients is $\mathbb{Z}_2$. Therefore, when
we change the orientation of the singular $1$-simplex in $(\Pi S^2,\Pi S^2_0)$
we obtain the same $1$-cycle.)
\par
It had been observed in [NR] that $v_1$ is homologous to $z_1$, and $v_2$ is homologous
to $z_2$, where $z_1$ and $z_2$ were defined in section 2.1. 
 Indeed, in order to see that $v_1$ is homologous to $z_1$, observe that
the maps of $[0,1]$ to $\Pi S^2$ representing these $1$-cycles are homotopic to each
other relatively to $\{0, 1\}$. The homotopy can be described as follows: After a
reparametrization of these paths and assuming that $m_0$ passes through $(0,-1,0)$ we can assume that
the circle $C_t$ corresponding to the image of each point $t\in (0,1)$ under the map representing
$z_1$ is tangent to $m_{t\pi}$  and $m_{-t\pi}$, and is contained inside the
digon $D_t$ formed by these two 
meridians and containing $m_0$ inside (see Fig. ~\ref{LS4}).

\begin{figure}[center] 
\includegraphics[scale=0.3]{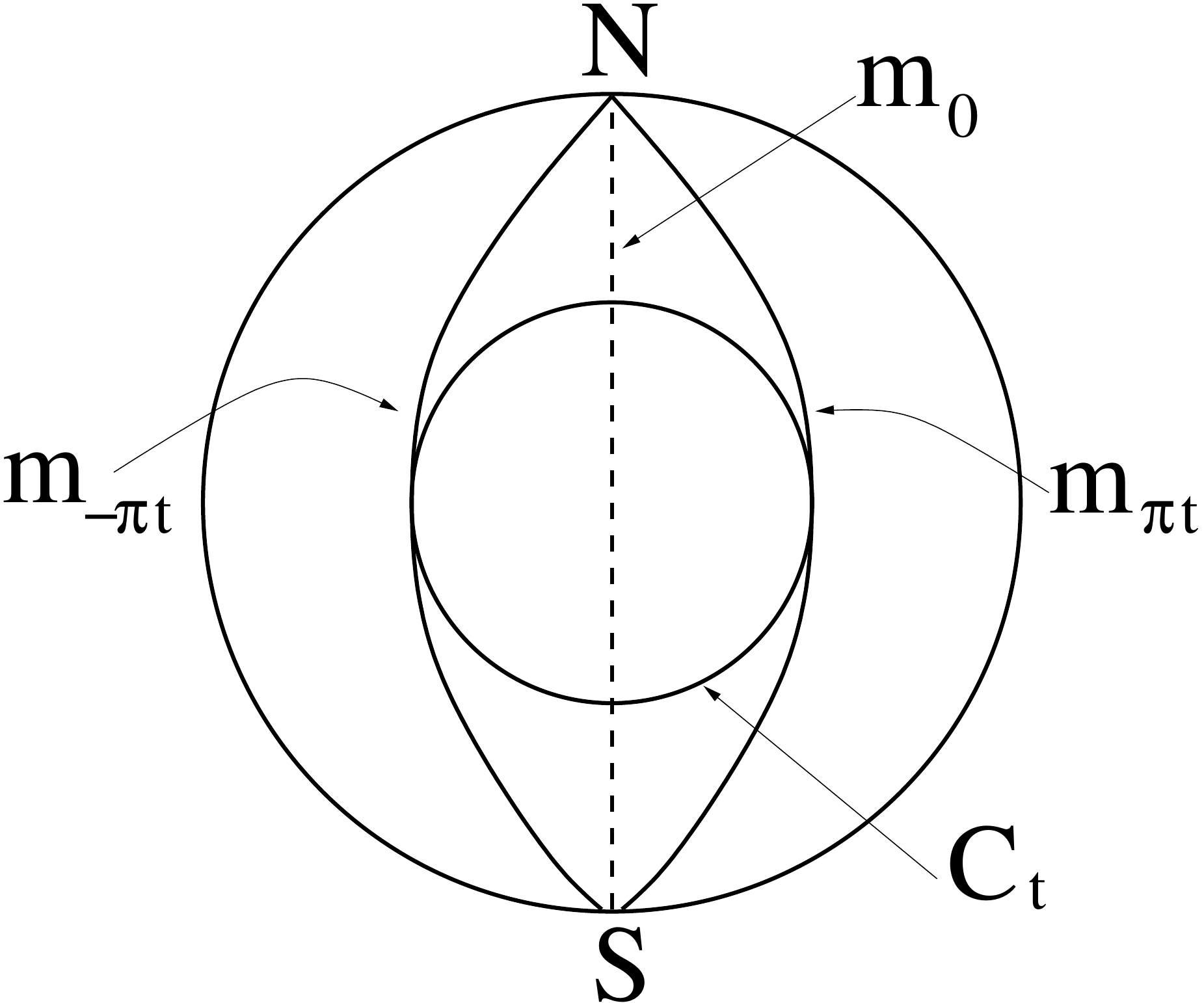} 
\caption{}
\label{LS2}
\end{figure}

For each $t$ the homotopy is merely a homotopy between $D_t$ to $C_t$ via simple curves that obviously can be made continuous with respect to $t$ (see
Fig. \ref{LS4}). Rotating this homotopy, we will obtain homotopy between the maps representing $2$-cycles $v_2$ and $z_2$.   

\begin{figure}[center] 
\includegraphics[scale=0.3]{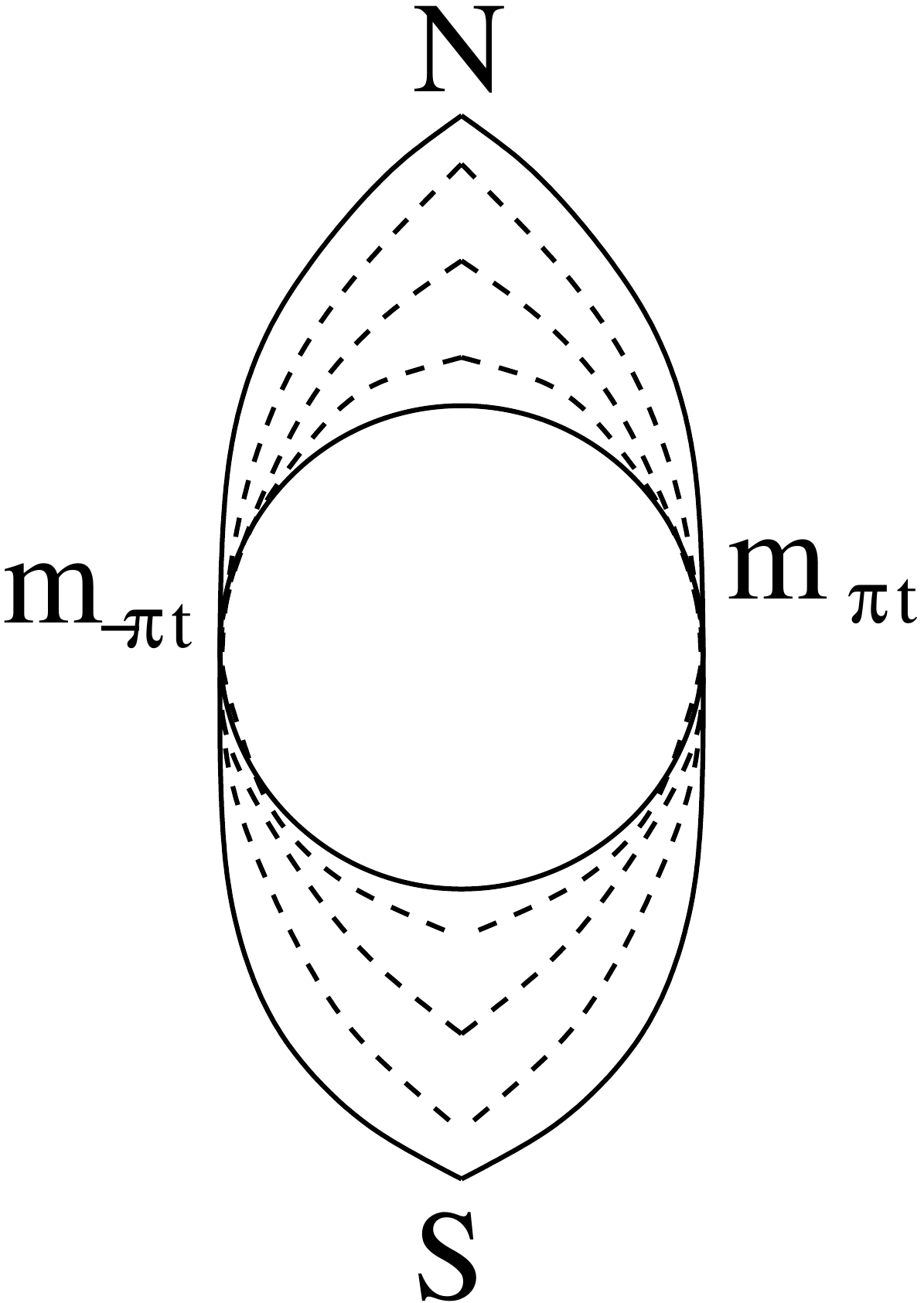} 
\caption{}
\label{LS4}
\end{figure}

{\bf 2.3. Cycles $v_3$ and $u_3$.} Again, $u_3$ will be the image
under the homomorphism induced by $f$ of a cycle $v_3\in Z_3(\Pi S^2,\ \Pi_0 S^2;\mathbb{Z}_2)$ that we are going to describe below.

Let $P$ denote all planes in $\mathbb{R}^3$
at a distance $\leq 1$ form the origin. 
Then $z_3 = \{p \cap S^2 | p \in P \}$ .
In other words, $z_3 \subset \Pi S^2$ is a family of all circles and points on $S^2$. 
Let $[z_3]$ denote the homology class of $z_3$ in $H_3((\Pi S^2,\Pi_0 S^2), \mathbb{Z}_2)$.

Suppose that a diffeomophism $f: S^2 \rightarrow (S^2,g)$ maps
each meridian to a curve of length $\leq L$ on $(S^2,g)$.
 We will construct a family $X \subset \Pi S^2$ of closed
curves, such that 1) $X$ will constitute a $3$-dimensional
relative cycle in $(\Pi S^2,\Pi_0 S^2;\mathbb{Z}_2)$;
2) $X \in [z_3]$ and the length of 
 $f(\gamma)$ does not exceed $\leq 4L +\epsilon$ 
 for every $\gamma \in X$. Then we define $v_3$ as the relative cycle
corresponding to $X$. In fact, below we sometimes will slightly
abuse this notation by writing $X$ for $v_3$ and considering
$X$ as an element of $Z_3(\Pi S^2,\Pi_0S^2;\mathbb{Z}_2)$.
 
 First, we will give a brief informal description of the construction of $X$.
 Consider a $1$-parametric family of planes in $\mathbb{R}^3$
 that pass through the South pole $S$ and contain
 some fixed line $l$ in the tangent space of $S^2$ at $S$.
 They intersect $S^2$ in a family of circles. Note that
these circles form a $1$-cycle homologous to $v_1$.
 Notice that this family of circles can be isotoped 
 to a family of curves each consisting of two 
 meridians on $S^2$. 
 On the other hand, consider a 2-parametric family of circles on $S^2$ that lie
 on a plane that is parallel to $l$.
Each circle $C$ in this family 
 can be sandwiched between two planes containing $l$
 and intersecting $C$ tangentially at its highest and lowest points.
(If the circle is in a horizontal plane, then the two points
of intersection are at the same height. However, even in this case the pair
of points of intersection is uniquely defined.)
 This defines a bijective correspondence between all circles in planes that 
 are parallel to $l$ and all pairs of circles on $S^2$
 that pass through $S$ and are tangent to $l$ at $S$ (see Fig. ~\ref{LS10}).

\begin{figure}[center] 
\includegraphics[scale=0.3]{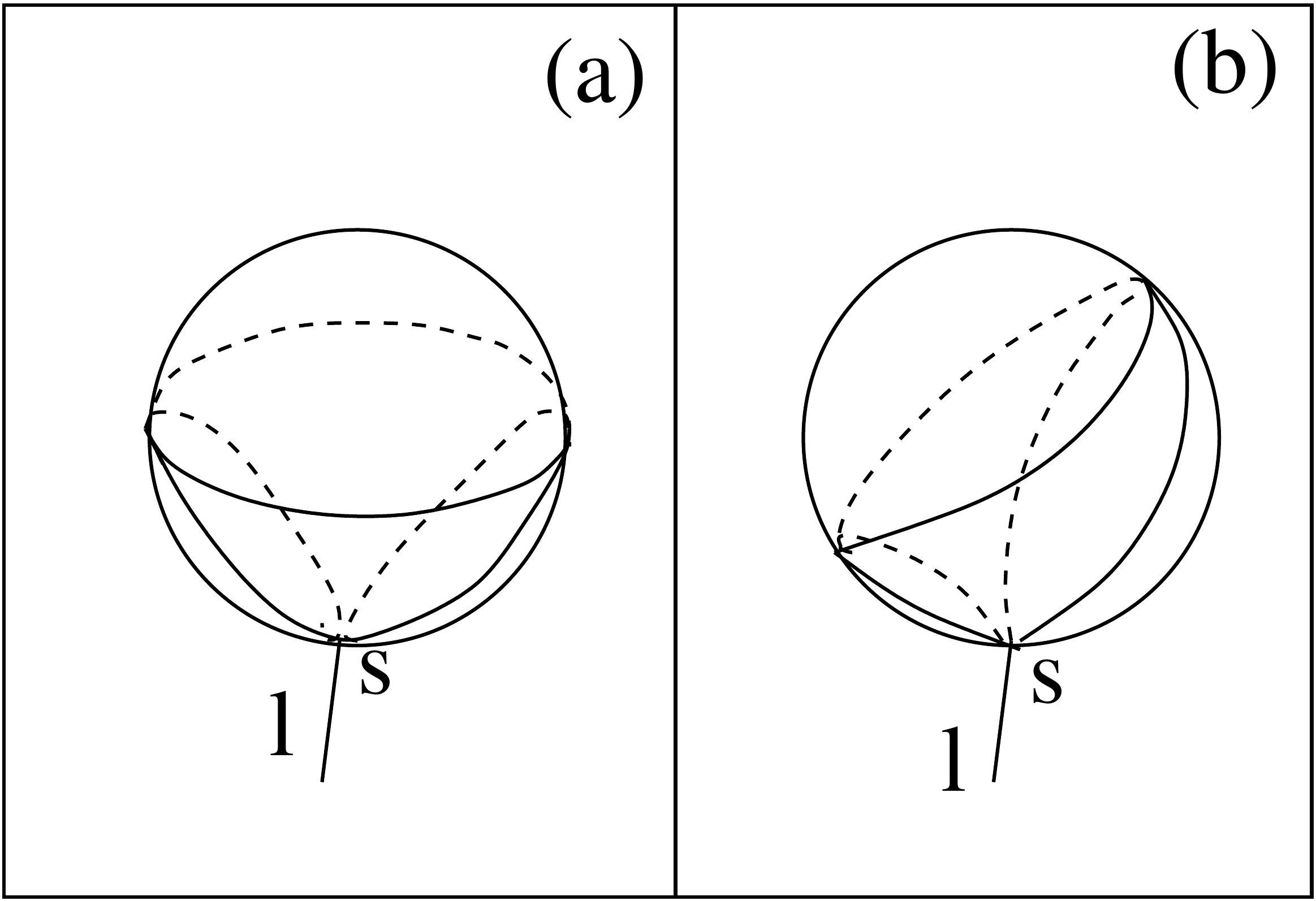} 
\caption{}
\label{LS10}
\end{figure}

 These pairs of circles in turn can be homotoped to a
 family of curves made out of four arcs of meridians (similarly to how it
was done in section 2.2).
 As we rotate $l$ by $\pi$, circles parallel to $l$ will yield the 
3-parametric family of all circles on $S^2$. We have a surjective continuous map from the set of all pairs of circles passing through $S$ and having a common
tangent at $S$ to the family of all circles on $S^2$. This map fails
to be injective only for horizontal circles on $S^2$, each of which will
have a $1$-parametric family of inverse images. Therefore, the cycle formed
by all considered pairs of circles will be homologous to $v_3$. Furthermore,
the cycle formed by all closed curves that go along four meridians
that form two pairs with common bissectoral planes is also homologous
to $v_3$. Finally, one can slightly perturb curves in this last cycle
to make them nonself-intersecting without significantly changing their length.
 
{\bf Cycle X of short curves.} Now we are going to give
a more detailed description of $X$.
We will construct a cycle $X (=v_3)$, such that each closed curve in $X$ consists
(up to a small perturbation) of at most four subarcs of the meridians on $S^2$.
This implies the above estimate.

We will describe the construction of $X$ in two steps. First, let us define a chain $X'$,
which has a non-empty boundary and curves with self-intersections where
the curve can touch itself but cannot cross itself.
Then we perturb the family so that it only consists of simple curves and connect the boundary 
through simple curves of bounded length to a $2$-cycle in $\Pi_0S^2$.
The image of this homotopy will be attached to the perturbed $X'$, and the
result will be $X$ that can now be regarded  as a relative $2$-cycle
in $(\Pi S^2,\Pi_0S^2)$.

\noindent
\textbf{Step 1. Defining $X'$}

Fix a meridian $m$ parametrized by the arclength.
Let $m^{\phi}$ denote a meridian making an angle $\phi$ with $m=m^0$,
$\phi \in [-\pi, \pi]$. For $t \in [0, \pi]$ define a family of closed curves

\[ \gamma_t ^m = \left\{ 
  \begin{array}{l l}
    m^0|_{[0,3t]} \cup -m^0|_{[0,3t]} & \quad \text{for $t \in [0, \frac{\pi}{3}]$}\\
    m^{3t - \pi} \cup -m^{\pi - 3t} & \quad \text{for $t \in [\frac{\pi}{3}, \frac{2 \pi}{3}]$}\\
    m^{\pi}|_{[0,-3(t-\pi)]} \cup -m^{\pi}|_{[0,-3(t-\pi)]} & \quad \text{for $t \in [\frac{2 \pi}{3}, \pi]$}
  \end{array} \right.\]

We identify the point curves $\gamma_0 ^m $ and $\gamma_{\pi} ^m$.
We perform a small perturbation that makes curves $\gamma_t$ 
simple and disjoint everywhere except at the South Pole.
The perturbation is illustrated in Figure \ref{fig:gamma_family}.
After the perturbation, $\gamma_{\pi/2} ^m$ is the big circle that intersects
South and North pole and is perpendicular to the meridian $m$.
For each value of the $z$ coordinate for $-1 < z < 1$
there are exactly two curves $\gamma_{\phi} ^m$ and $\gamma_{\pi-\phi} ^m$
that are touching the $z=const$ plane at their highest point,
which we will denote by $a(\phi)$
Moreover, curves $\gamma_{\phi} ^m$ and $\gamma_{\pi-\phi} ^m$ are mirror images of each other
reflected along $\gamma_{\pi/2}$.

\begin{figure}[center] 
\includegraphics[scale=0.4]{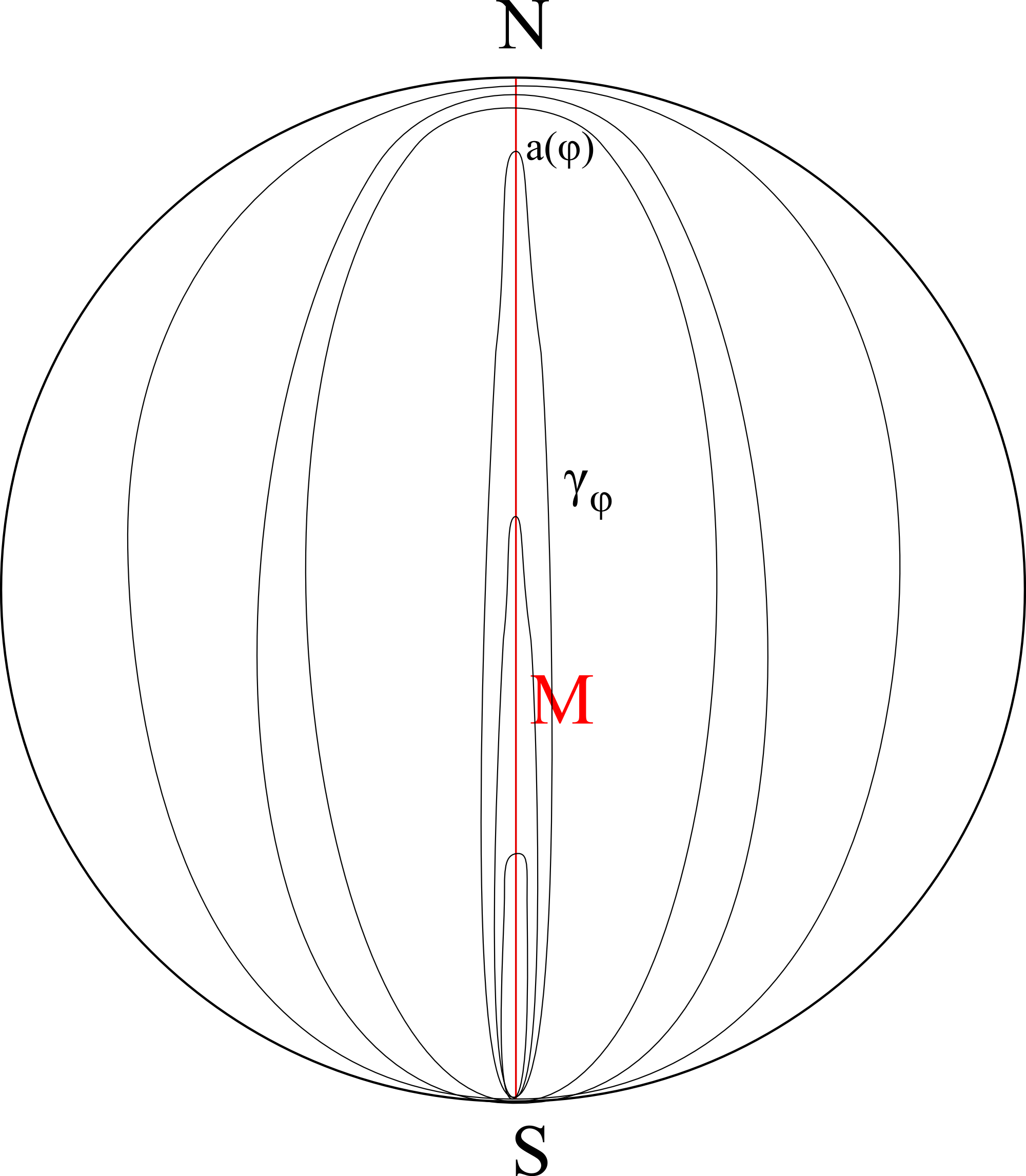} 
\caption{One-parametric family $\gamma_t$} 
\label{fig:gamma_family}
\end{figure}
  
Consider triangle $T= \{(\phi_1,\phi_2)| \phi_1 \geq \phi_2\} \subset [0, \pi]^2$.
For each $(\phi_1,\phi_2) \in T$ let $\gamma_{(\phi_1,\phi_2)}=\gamma_{\phi_1}^m \cup \gamma_{\phi_2}^m$.
We connect the endpoints of $\gamma_{\phi_1}$ and $\gamma_{\phi_2}$
in such a way that after a small perturbation $\gamma_{(\phi_1,\phi_2)}$ can be made simple
(see Fig. \ref{fig:perturbation at South Pole}).

Two sides of $T$, $a=\{(x,0)| 0\leq x \leq \pi \}$ and $b=\{(\pi,y)| 0\leq y \leq \pi \}$,
parametrize the same family of curves, so we glue them together accordingly. 
Note that $(0,0)$, $(0, \pi)$ and $(\pi,\pi)$ are all identified into one point.
As a result we obtain a family of curves parametrized
by a M\"obius band (in the space $\Pi S^2$), where the boundary is formed by all
$c=\{(\phi,\phi)| \phi \subset [0, \pi]\}$,
consisting of curves $\gamma_{\phi}^m$ traversed twice.
See Figure \ref{fig:T}.

\begin{figure}[center] 
\includegraphics[scale=0.4]{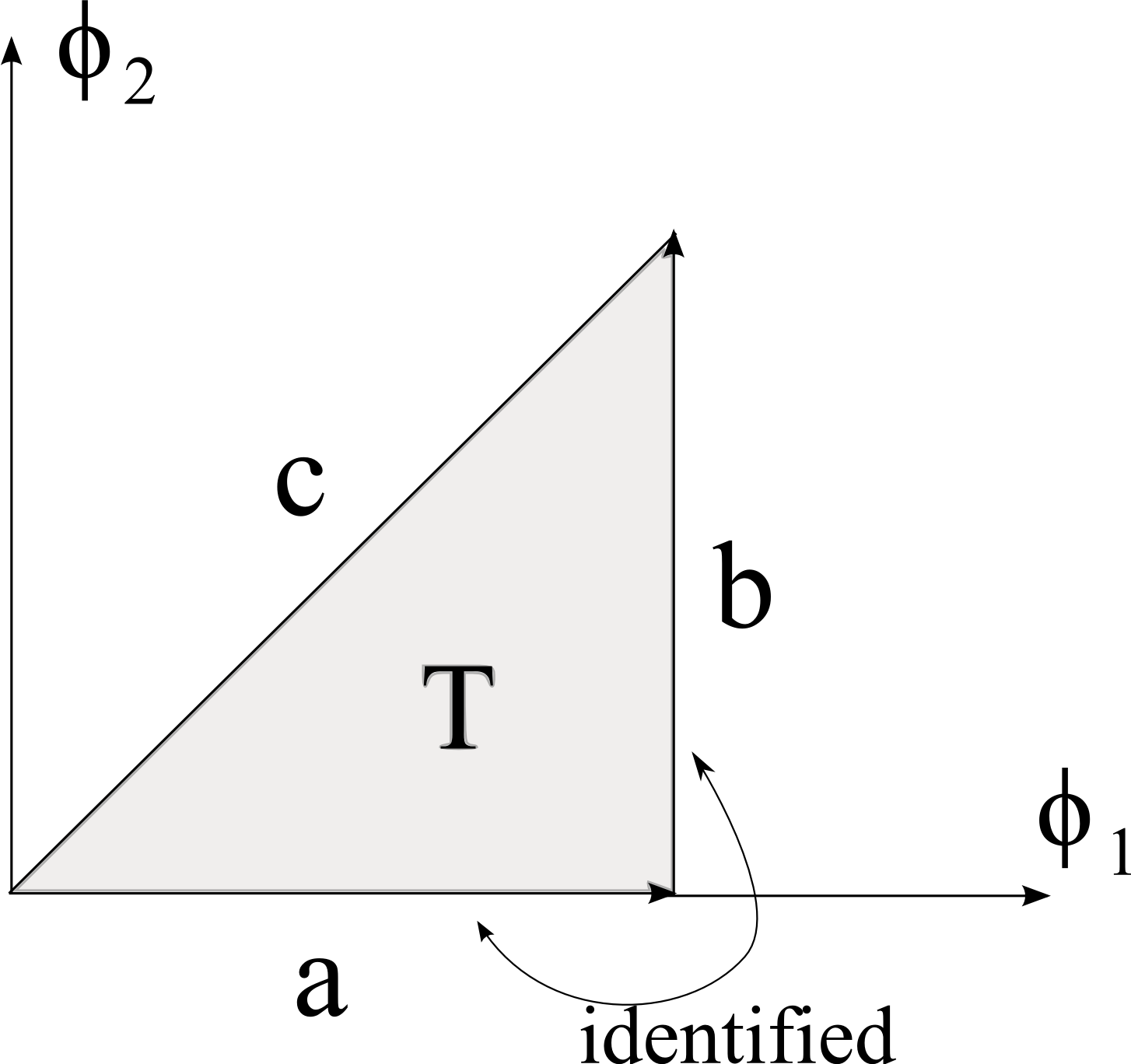} 
\caption{Triangle $T$ parametrizing family $X_m$} 
\label{fig:T}
\end{figure}

\begin{figure}[center] 
\includegraphics[scale=0.3]{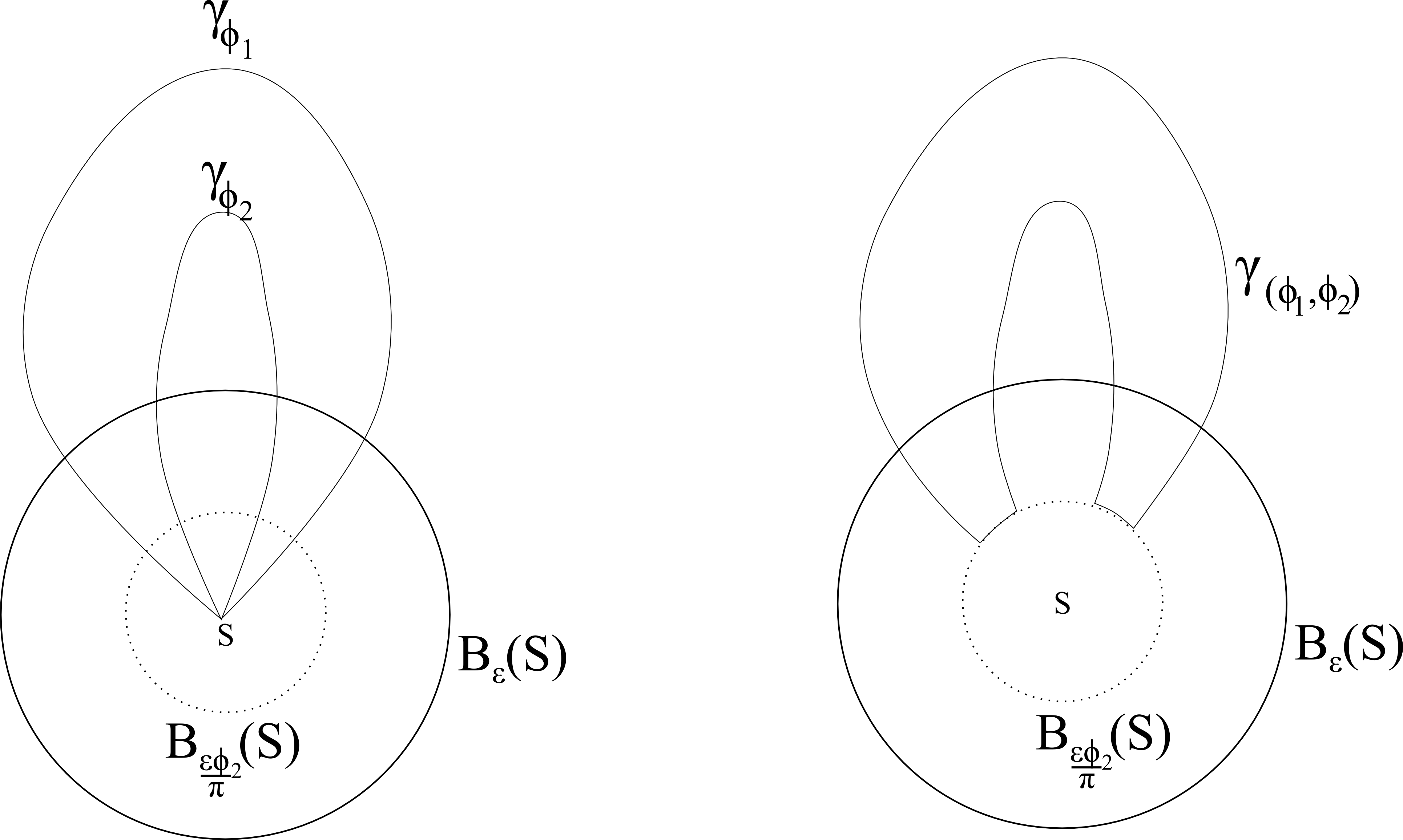} 
\caption{Perturbation of $\gamma_{(\phi_1,\phi_2)}$ at the South Pole} 
\label{fig:perturbation at South Pole}
\end{figure}

Denote by $X_m$ the family of curves parametrized by $T / a \sim b$.
Rotating $m$ by an angle $\pi$ and identifying
the curves of $X_m$ and $X_{m^{\pi}}$
we obtain a 3-parametric family $X'$.

\noindent
\textbf{Step 2. Contracting the boundary of $X'$ and removing self-intersections} 

Closed curves in $X'$ do not have transverse 
self-intersections, but some of them have arcs traced multiple times and 
non-transverse intersections at the South Pole.

First we define a perturbation that gets rid of self-intersections at the South Pole.
Choose a small $\epsilon_0-$ball $B_{\epsilon_0}(S)$ around the South Pole.
Consider $\gamma_{(\phi_1,\phi_2)} \in X_m$ ($\phi_1 \geq \phi_2$)
and let $\epsilon(\phi_2)=\frac{(z(a(\phi_2))+1) \epsilon_0}{2}$.
We replace the intersection of $\gamma_{(\phi_1,\phi_2)}$ 
with $B_{\epsilon(\phi_2)} (S) \subset B_{\epsilon_0}(S)$ by two arcs
on the boundary circle of $B_{\epsilon (\phi_2)} (S)$
as on Figure \ref{fig:perturbation at South Pole}.
One can check that this defines a continuous mapping on $X'$.

Now we get rid of double traced arcs and contract the boundary of chain $X'$.
Consider a one-parametric family of curves $L_m = \{\gamma_{(\phi_1,\phi_2)}|
\phi_1-\phi_2= \epsilon \} \subset X_m$ for some small $\epsilon>0$. 
Each curve $\gamma_{(\phi_1,\phi_2)} \in L_m$ bounds a narrow disc around $\gamma_{\theta_1}$
for $\theta_1 = \phi_1 + \epsilon/2$. We can contract each $\gamma_{(\phi_1,\phi_2)}$
to a point $a(\theta_1)$ (the point of maximal z-coordinate on $\gamma_{\theta_1}$)
in this discs through simple closed curves
as on Figure \ref{fig:contraction}. This is continuous for all families $L_m$
as we vary the meridian $m$. We replace the part of $X'$ formed by all
curves $\{\gamma_{(\phi_1,\phi_2)}\vert \phi_1-\phi_2<\epsilon\}$ by the
image of this contracting homotopy. All points $\gamma_{(\phi,\phi)}$
on the boundary
of the M\"obius band will be mapped to the trivial curves $\{a(\phi)\}$. This completes our construction of $X$.

\begin{figure}[center] 
\includegraphics[scale=0.5]{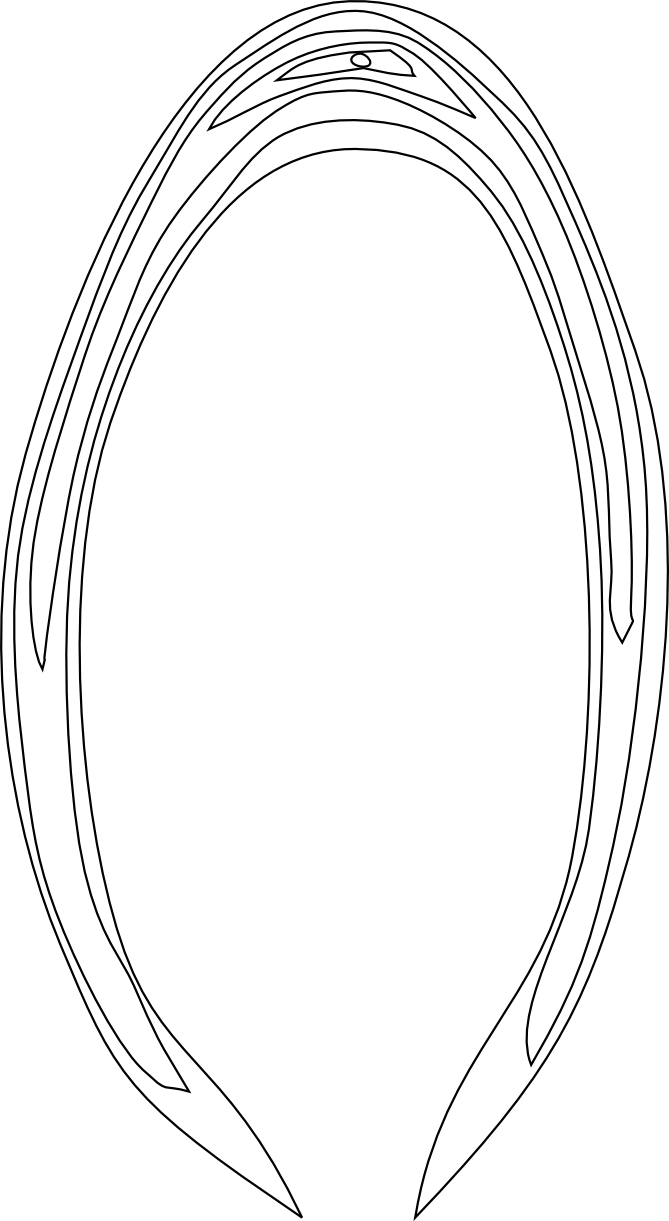} 
\caption{Contracting a curve near the boundary of $X_m$} 
\label{fig:contraction}
\end{figure}

{\bf 2.3.2 Proof that $X \in [z_3]$} We claim that $X$ represents
the same homology class in 
$H_3((\Pi S^2,\Pi_0 S^2), \mathbb{Z}_2)$ as the family $z_3$
of all circles on $S^2$. (In this subsection $z_3$ will
sometimes be regarded as a subset of $\Pi S^2$ and sometimes
as the corresponding relative $3$-cycle.)
To prove this we will construct a homotopy 
 $\sigma_t: X \rightarrow  \Pi S^2$
with $\sigma_0$ the identity map and
$\sigma_1$ coinciding with $z_3$
on a dense open subset. Fixing a CW structure on 
cycles $\sigma_1$ and $z_3$ we will observe that
the two maps coincide on the interior 
of their 3-cells. In particular, it will follow
that they represent the same homology class.

\begin{figure}[center] 
\includegraphics[scale=0.3]{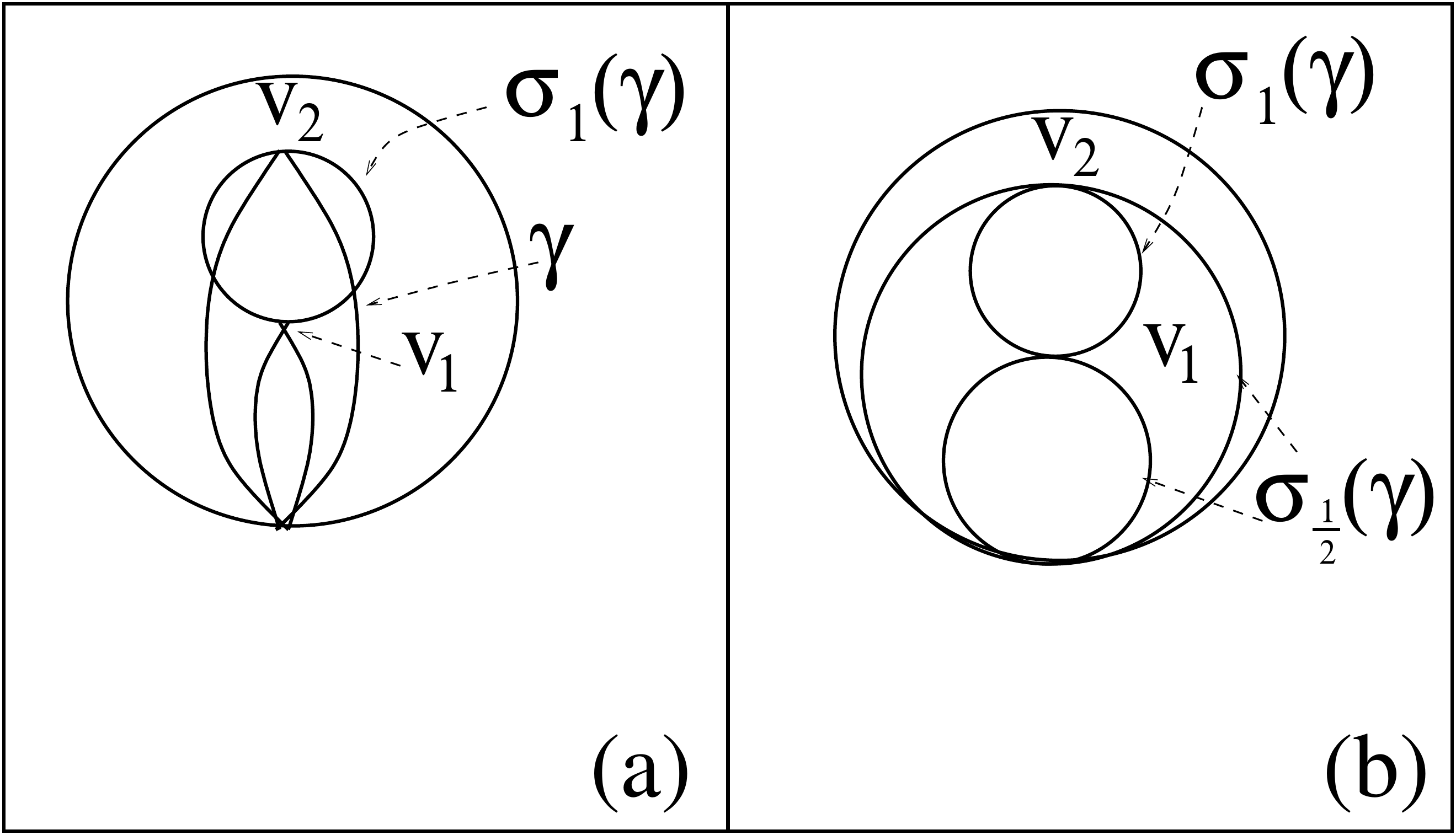} 
\caption{}
\label{LS11}
\end{figure}

\noindent
\textbf{Description of $\sigma_1(X)$}
(See Fig. ~\ref{LS11}). First, we describe $\sigma_1$.
If $\gamma \in X$ is a point curve we set $\sigma_1(\gamma)= \gamma$.
Otherwise, by construction, for each curve $\gamma \in X$
there is unique big circle $C$  that passes through the South pole
and the point, where $\gamma$ attains its maximal $z$-value. Note that $\gamma$
is symmetric under the reflection along $C$.
The curve $\gamma$ intersects $C$ at two points $v_1$ and $v_2$.
We set $\sigma(\gamma)$ to be the circle on $S^2$ that passes
through $v_1$ and $v_2$ and whose center lies on $C$.
Clearly $\sigma$ is continuous and maps $X$ onto $z_3$.
It is not injective. Note that each longitudinal 
circle has a one-parametric family of pre-images under $\sigma_1$.
 
The homotopy between the identity map on $X$ and $\sigma_1$
will be constructed in two stages. First we homotop
$X$ to a family $Y=\sigma_{1/2}(X)$ of curves obtained from two circles
that pass through the South pole and whose centres lie on the same
meridional big circle. We described this family at the beginning of
section 2.3 as a 
family of curves obtained from a pair of circles that 
have a common tangent line at the South pole.
Then we will homotop $Y$ to $\sigma_1(X)$.

\noindent
\textbf{Homotopy from $X$ to $Y$} 
Given a curve $\gamma_{\theta} ^m$ define 
$\overline{\gamma}_{\theta} ^m$ to be the circle that passes through
the South Pole and
$a(\theta)$ (the point where $\gamma_{\theta}^m$ attains its maximal $z$ value)
and whose center lies on $m \cup m^{\pi}$.
For each $\gamma_{(\phi_1,\phi_2)} \in X_m$
we set $\overline{\gamma}_{(\phi_1,\phi_2)} = \overline{\gamma}_{\phi_1} ^m \cup \overline{\gamma}_{\phi_1} ^m$.

We perform perturbations that remove self-intersections 
at the South Pole and doubly traced curves in exactly the same way 
as we described above for the family $X$.
This correspondence is a homeomorphism and it is easy to check that there
exists an isotopy from $X$ to $\sigma_{1/2}(X)=Y$ through simple curves that
gradually ``fattens" each curve $\gamma_\phi^m$ into the corresponding circle
$\overline{\gamma}_\phi^m$.

\noindent
\textbf{Homotopy from $Y$ to $\sigma_1(X)$}
 
Let $\gamma \in Y$.
 We will define a homotopy 
$\gamma_t$, $t \in [-1,1]$, that starts
on $\gamma$ and ends on $\overline{\gamma}=\sigma_1(\sigma_{1/2} ^{-1} (\gamma))$.
It will be clear from the definition that
the homotopy depends continuously
on $\gamma \in Y$.

As before let $S$ be the big circle of symmetry of 
$\gamma$ and $v_1$ and $v_2$ be points of intersection
of $\gamma$ with $S$.
Let $D_1$ be the disc of $S^2 \setminus \gamma$
that contains $\overline{\gamma}$. 
Let $D_2$ denote the disc of $S^2 \setminus \overline{\gamma}$ that does not
contain $\gamma$. Note that $D_2 \subset D_1$.

Let $a_1$ and $a_2$ be the $z-$coordinates of points $v_1$ and $v_2$
correspondingly and assume without any loss of generality that
$a_1 \leq a_2$.

For $t$ so small that the plane $\{z=t\}$
does not intersect $\gamma$ we set $\gamma_t=\gamma$.
Suppose $\{z=t\}$ intersects $\gamma$, but $t \leq a_1$.  
Let $\gamma^{\geq t}$ denote the intersection 
of the curve $\gamma$ with the halfspace $\{z \geq t\}$.
Let $\alpha_{t}$ denote the intersection of the disc
$D_1$ with the plane $\{ z = t \}$. Note that
$\alpha_{t}$ consists of two arcs connecting the endpoints of 
$\gamma^{\geq t}$.
We define $\gamma_t = \gamma^{\geq t} \cup \alpha_t$.

Suppose $a_1 \leq t \leq a_2$. 
Define $\beta_{t}$ to be the intersection of
$D_1 \setminus D_2$ with the plane $\{ z = t \}$. 
Let $c^{\leq t}$ denote the intersection of $\overline{\gamma}$
with $\{z \leq t\}$.
We define $\gamma_t = \gamma^{\geq t} \cup \beta_t \cup c^{\leq t}$.

For $t \geq a_2$ we set $\gamma_t = \overline{\gamma}$.
 
As we vary the initial curve $\gamma \in Y$,
the 1-parametric family obtained in this way changes continuously.
This completes our definition of homotopy $\sigma$. And, finally,
we observe that:
 
\begin{lemma} \label{YZ}
$[\sigma_1(X)] = [z_3] \in H_3((\Pi S^2,\Pi_0 S^2), \mathbb{Z}_2)$.
\end{lemma}

\begin{proof}
Let $L = \{l_{s} \} \subset z_3$, $-1 \leq s \leq 1$, denote 
all longitudinal circles $z = s$.
We fix a CW structure on $z_3$ so that $L$ is contained in the 1-skeleton
and a CW structure on $X$ so that $\sigma_1^{-1}(L)$ is contained in the
2-skeleton.
Note that $\sigma$ maps $X \setminus \sigma_1^{-1}(L)$ homeomorphically
onto $z_3-L$. It follows that the simplicial $3$-chain corresponding to $z_3$
coincides with the image
of a $3$-dimensional chain with $\mathbb{Z}_2$-coefficients 
respresenting the fundamental class of $X$
under the homomorphism induced by $\sigma$. 
\end{proof}


{\bf 2.4. Meridional slicing of $M$ and the lengths of three simple
periodic geodesics.}

The discussion in section 2.1 and the description of cycles $u_1, u_2$
and $u_3$ in sections 3.2 and 3.3 imply the following proposition:

\begin{Pro} \label {Proposition2.2} 
Assume that $p,q$ are (not necessarily distinct) points of $M$, and $\phi$ is a continuous map of $S^1$ into the space
of piecewise smooth curves connecting $p$ and $q$ in $M$ such that for each point $r\in M$ different from $p$ and $q$
exists unique $t\in S^1$ such that $\phi(t)$ passes through $r$. Assume also that for each $t\in S^1$ the length
of $\phi(t)$ does not exceed $L$. (In other words, $\phi$ is a slicing of $M$ into picewise smooth curves of length $\leq L$
connecting $p$ and $q$.) Assume also, that $\min_{t\in S^1}$ length$(\phi(t))\leq L_0$ for some $L_0\leq L$. Then
there exist three distinct simple periodic geodesics on $M$ of non-zero index with lengths not exceeding, respectively, $L_0+L$, $2L$ and $4L$.
\end{Pro}

\begin{Pf}{Proof of Theorem 1.2}
Combining this proposition with Theorem 1.3 B from [LNR] we immediately obtain a proof of Theorem 1.2. (Theorem 1.3B from [LNR] asserts that for any Riemannian
$2$-sphere $M$ of area $A$ and diameter $d$ and any point $p\in M$ there exists
a slicing of $M$ into nonself-intersecting loops based at $p$ of length
$\leq 200d\max\{1,\ln({\sqrt{A}\over d})\}$ so that these loops instersect
each other only at $p$.)
\end{Pf}

Now we are going to prove Theorem 1.3.

\begin{Pf}{Proof of Theorem 1.3}

We prove Theorem 1.3 by combining Proposition 2.2
with results of [LNR].

We consider two cases.

Case 1. Suppose the diameter $d$ of the 2-sphere $M$
satisfies $d \leq \sqrt{3} \sqrt{A}$.
By Theorem 1.3 from [LNR] there exists
a slicing of $M$ into nonself-intersecting based loops of length
bounded by $664 \sqrt{A} +2 d < 700 \sqrt{A}$. It follows 
from the proof that these loops do not intersect except at their 
base point. 

Case 2. Suppose $d > \sqrt{3} \sqrt{A}$.
Let $p$ and $q$ be two points on $M$ 
that are at the maximal distance $d$ from each other.
Let $S_p(r) = \{x|dist(x,p)=r\}$ be the geodesic sphere
around $p$ of radius $r$.
For almost every $r$ the set $S_r(p)$ is a 
finite collection of simple closed curves on $M$.
By coarea inequality we can find $r$, 
$|r - d/2| \leq \frac{\sqrt{3}}{2}\sqrt{A}$,
such that $S_r(p)$ is a collection of simple
closed curves of total length at most $\frac{\sqrt{3}}{3} \sqrt{A}$.

Let $\gamma$ be a connected component of 
$S_r(p)$ that separates $p$ and $q$.
In other words, $M \setminus \gamma = D_1 \sqcup D_2$
with $p \in D_1$ and $q \in D_2$. 

Next we will construct a slicing in each of the discs $D_i$
using Theorem 1.6 C from [LNR].
For a Riemannian $2$-disc $D$ define 
$d_D = \max \{dist(p,x) | p \in \partial D, x \in D\}$. 
Note that $d_D$ is bounded from above by the diameter of $D$.
Theorem 1.6 C from [LNR] asserts that for any two points $x$ and $y$
on the boundary of $D$ there exists a path homotopy between two arcs of
$\partial D$ with endpoints $x$ and $y$, such that the lengths of all
curves in the homotopy are less than 
$2L + 686 \sqrt{A} + 2d_D$, where $L$ denotes the length of $\partial D$.
Moreover, it follows from the proof that the curves have no 
self-intersections and intersect each other only 
at the endpoints.

We claim that $d_{D_i} \leq d/2 + \frac{3 \sqrt{3}+3}{2} \sqrt{A}$.
Indeed, let $l$ be a minimizing geodesic from $p$ to $q$.
Since $\gamma$ is a connected component of a geodesic sphere,
$l$ will have a unique intersection point with $\gamma$.
Let $a$ denote this intersection point, $l_1$ be the arc of $l$
from $p$ to $a$ and $l_2$ be the arc of $l$ from $a$ to $q$. 

Let $p_1 \in D_1$ and $a_1 \in \gamma$ be such that
$dist(p_1,a_1) = d_{D_1}$. Using the triangle inequality
we obtain the following inequalities:
$$d \geq dist(p_1,q) \geq dist(p_1,\gamma) +dist(\gamma,q)
\geq d_{D_1} - \frac{length(\gamma)}{2} + length (l_2).$$

By definition of $\gamma$ we have $length(\gamma) \leq \frac{\sqrt{3}}{3} \sqrt{A}$
and $length(l_2) \geq d/2 - \frac{\sqrt{3}}{2} \sqrt{A}$.
We obtain the desired bound 
$d_{D_1} \leq d/2 + \frac{3 \sqrt{3}+3}{2} \sqrt{A}$.
For disc $D_2$ the proof is identical.
Choosing any two points $x$ and $y$ on $\gamma$ and applying 
Theorem 1.6 C from [LNR] to $D_1$ and $D_2$
we obtain a meridional slicing of $M$ by curves
connecting $x$ and $y$ with length bounded by
$d + 700 \sqrt{A}$.

By Proposition 2.2 we obtain three distinct simple closed curves
of positive index with
the length of the first curve bounded
by $d + 700 \sqrt{A}$, the length of the second curve bounded by
$2d + 1400 \sqrt{A}$, and the length of the third curve bounded
by $4d + 2800 \sqrt{A}$.

\end{Pf}

\section {Slicing of $M$ into short curves in the absence of short simple periodic geodesics.}

{\bf 3.1. Geodesic segments beween a pair of the
most distant points of $M$.} In this subsection we review some ideas
of C.B. Croke from [Cr]. Let $p$, $q$ be two points on $M$ such that the distance between them is the maximal possible (and is equal to $d$).
In [Cr] Croke observed that either there exists a periodic geodesic on $M$ of length $\leq 2d$, or there exists
a sweep-out of $M$ by a continuous $1$-parametric family of curves
connecting $p$ and $q$ of length $\leq 8d$, leading to an upper bound $9d$
for the length of the shortest non-trivial periodic geodesic.
\par
The argument of Croke starts from the observation that $p$ and $q$ can
be connected by finitely many minimizing geodesics that divide $M$ into
geodesic digons with angles at $p$ and $q$ that do not exceed $\pi$.
This fact easily follows from Berger's lemma. Indeed, if this is not so, then
one
would be able to increase the distance between $p$ and $q$ by moving one
of these two points inside a geodesic digon that has the angle adjacent to this point greater than $\pi$ and is free
of minimizing geodesics.
\par
Now we would like to interrupt our exposition to review several definitions and facts. A simple closed curve on $M$ is called {\it convex}
with respect to one of two discs $D$ bounded by this curve if each sufficiently short geodesic in $M$ connecting two points of $\gamma$ is
contained in $D$. Sometimes we will say that a closed curve is convex without mentioning the relevant disc,
if it is clear from the context which of the two discs we have in mind.
A homotopy between $\gamma$ and another closed curve $\beta$ in a $2$-disc
$D$ is called {\it monotone}
if all curves $\gamma_t$, $t\in [0,1], \gamma_0=\gamma, \gamma_1=\beta$, in this homotopy are simple, and for
each $t_1, t_2$ such that $t_1<t_2$ the closed disc bounded by $\gamma_{t_2}$ in $D$ is contained in the closed
disc bounded by $\gamma_{t_1}$ in $D$.
In other words,
``monotone" means here that the disc inside the digon bounded by a
curve at a later moment of time is contained in the disc bounded by a curve
at an earlier moment of time.
A {\it path homotopy} betweeen
two curves connecting a pair of points $u$ and $v$ is a homotopy that
passes through curves connecting $u$ and $v$.
Further, we will be using the Birkhoff curve-shortening process which is explained in [Cr] and has the following properties (also explained in [Cr]):
For each closed curve $\gamma$ the Birkhoff curve-shortening process produces
a length non-increasing homotopy that connects $\gamma$ with either a point or a periodic geodesic. If the geodesic to which Birkhoff curve-shortening process
converges has index $>0$, we can perturb it decreasing its length, and
continue using the Birkhoff curve-shortening process. Therefore,
without any loss of generality we can assume that the process converges either
to a point or to a non-trivial periodic geodesic of index $0$.
\par
Birkhoff curve-shortening process depends on a
small parameter $\epsilon>0$. If the
initial curve is simple and convex to a disc $D$, then for all sufficiently
small positive values of this parameter
all intermediate curves will be simple, convex and contained in $D$;
the resulting homotopy will be monotone.
Other features of this well-known process
are not important for us here, and, in fact, this process can be replaced by other curve-shortening processes with similar
properties such as Grayson's curvature flow ([Gr], or a process devised by J. Hass and P. Scott ([HS]).

\par
The next observation of Croke is that these geodesic digons with 
the vertices at $p$ and $q$ and angles 
not exceeding $\pi$ are convex, and, therefore, 
the Birkhoff curve-shortening process will shrink the boundary of 
each digon inside the digon in a monotone way (for a sufficiently
small value of the parameter). The process will terminate
either at a point or at a non-trivial periodic geodesic of length not exceeding $2d$ inside the digon. If there are no such periodic geodesics,
the boundary of the digon contracts to a point by a length non-increasing
monotone homotopy. 

Then Croke observes that this homotopy contracting the boundary of the digon can be transformed into a path homotopy
between two geodesics forming the boundary that passes through paths
of length $\leq 8d$ containing in the digon.
All those homotopies
for various digons can be combined into a sweep-out of $M$ into paths
between $p$ and $q$ of length $\leq 8d$. (See Lemma 3.1 below
for an explicit statement in this direction that also
yields a better bound than $8d$.) Some of these paths are the
original minimizing geodesics between $p$ and $q$ of length $\leq d$. Choosing one of these
minimizing geodesics and attaching it to all the paths one obtains a
$1$-parametric family of loops of length $\leq 9d$ based at $p$ and representing a non-trivial homology class of the space of free loops of $M$. 
Now the standard Morse-theoretic proof of the existence of periodic geodesics
implies the existence of a non-trivial periodic geodesic on $M$ of length $\leq 9d$.
\par

{\bf 3.2. Some observations of M. Maeda.} M. Maeda improved this argument in two directions ([M]). Let $D$ be one of the convex
digons formed by two minimizing geodesics between $p$ and $q$. Let $x\in D$
be a point where $f(x)=dist(x, \partial D)$ attains its maximum. Denote
this maximal value by $r$. First, he observed that $r$ can exceed ${d\over 2}$ for at most
one of the digons (and, of course, in this case it still does not exceed $d$). The reason
is simple: if there are two digons $D_1,\ D_2$ and points $x_1\in D_1$, $x_2\in D_2$ such that 
$dist(x_i,\partial D_i)>{d\over 2}$ for $i=1,2$, then $dist (x_1,x_2)>d$ which is impossible.

To explain the second observation of Maeda  let
$\Delta_t,\ t\in [0,1]$, denote (convex)
curves obtained from $\Delta_0=\partial D$ during an application of the Birkhoff curve-shortening
flow ending at a point $u$ (so that $\Delta_1={u}$). Let $\gamma$ 
denote a minimal
geodesic connecting a closest to $u$ point $z$ on $\Delta_0$ with $u$,
then for each $t$ there is exactly one point of intersection of $\gamma$
with $\Delta_t$, and this point of intersection moves continuously
and monotonically
from $z$ to $u$.

Indeed, otherwise, $\gamma$ might exit at a second point of intersection and will need to intersect
$\Delta_t$ again to return to the disc bounded by $\Delta_t$ and containing $q$.
As the result, 
$\gamma$ will touch from inside one of
the piecewise geodesic curves $\Delta_t$ which is impossible because of their
convexity. Indeed, assume that $\gamma$ touches a curve $\Delta_s$ from
inside at a point $w$. (When we say that $\gamma$ touches $\Delta_s$ from inside at $w$, we
mean that an open arc of $\gamma$ containing $w$ stays
inside the closed disc bounded by $\Delta_s$ and containing $q$.) 
Take two points $w_1, w_2\in \Delta_s$ that lie on opposite sides
of $w$ on $\Delta_s$ and such that the length of the arc of $\Delta_s$ that
connects them is very small. Then the convexity of $\Delta_s$ implies that
the minimal geodesic connecting $w_1$ and $w_2$ must intersect $\gamma$
at at least two distinct points which contradicts the minimality
of $\gamma$. 

The second option is 
$\gamma$ reaching $\Delta_t$ at a moment of time $s_1$, then going along
$\Delta_t$ until a moment of time $s_2>s_1$, then entering
the open disc bounded by $\Delta_t$ (and not returning to $\Delta_t$ anymore).
To exclude this possibility consider $w^*=\gamma(s_2)$. (In other words,
$w^*$ is the ``last" point of $\gamma$ on $\Delta_t$.) Take two points
$w_1^*$ and $w_2^*$ on $\Delta_t$ that are very close to $w^*$ and lie on
the opposite sides of $w^*$. (So, one of them is in the image of $\tau$,
and the other is not. Now, the convexity of $\Delta_t$ imples that the minimal
geodesic between $w_1^*$ and $w_2^*$ must intersect $\gamma$ at a point
different from $w_1^*$ and $w_2^*$, and we again obtain a contradiction
with the minimality of $\gamma$.

Now we know that $\gamma$ intersect $\Delta_t$
at exactly one point. An obvious argument shows that this point of intersection
moves continuously and monotonously along $\gamma$.

Denote the segment of $\gamma$ from $z$ to the
point of intersection of $\gamma$ with $\Delta_t$ by $\gamma_t$. Maeda considers loops based at $z$
formed by first following $\gamma_t$, then $\Delta_t$ , then $\gamma_t$ travelled in the opposite direction, and concludes that they provide a way to contract
$\Delta_0$ to a point through loops based at $z$ of length less than or equal
to $2r+2d$, where $r$ denotes $\max_{x\in D_t} dist (x,\Delta_t)$. Here
$D_t$ denote the disc bounded by $\Delta_t$ containing $u$. (One also needs
to cancel the the loop obtained by foloowing $\gamma$ from $z$ to $u$,
and then back to $z$ along itself at the end of this homotopy.)
One can use the homotopy contracting $\Delta_0$ as a loop based at $z$ to construct a path homotopy between one side of $\Delta_0$ to the other (keeping
the endpoints fixed during the homotopy) (see Fig. \ref{LS12} ). The length of curves in this homotopy
is bounded by $2r+3d\leq 4d$, if $r\leq {d\over 2}$, and by $5d$ in the
general case (as $r$ cannot exceed $d$).  

\begin{figure}[center] 
\includegraphics[scale=0.3]{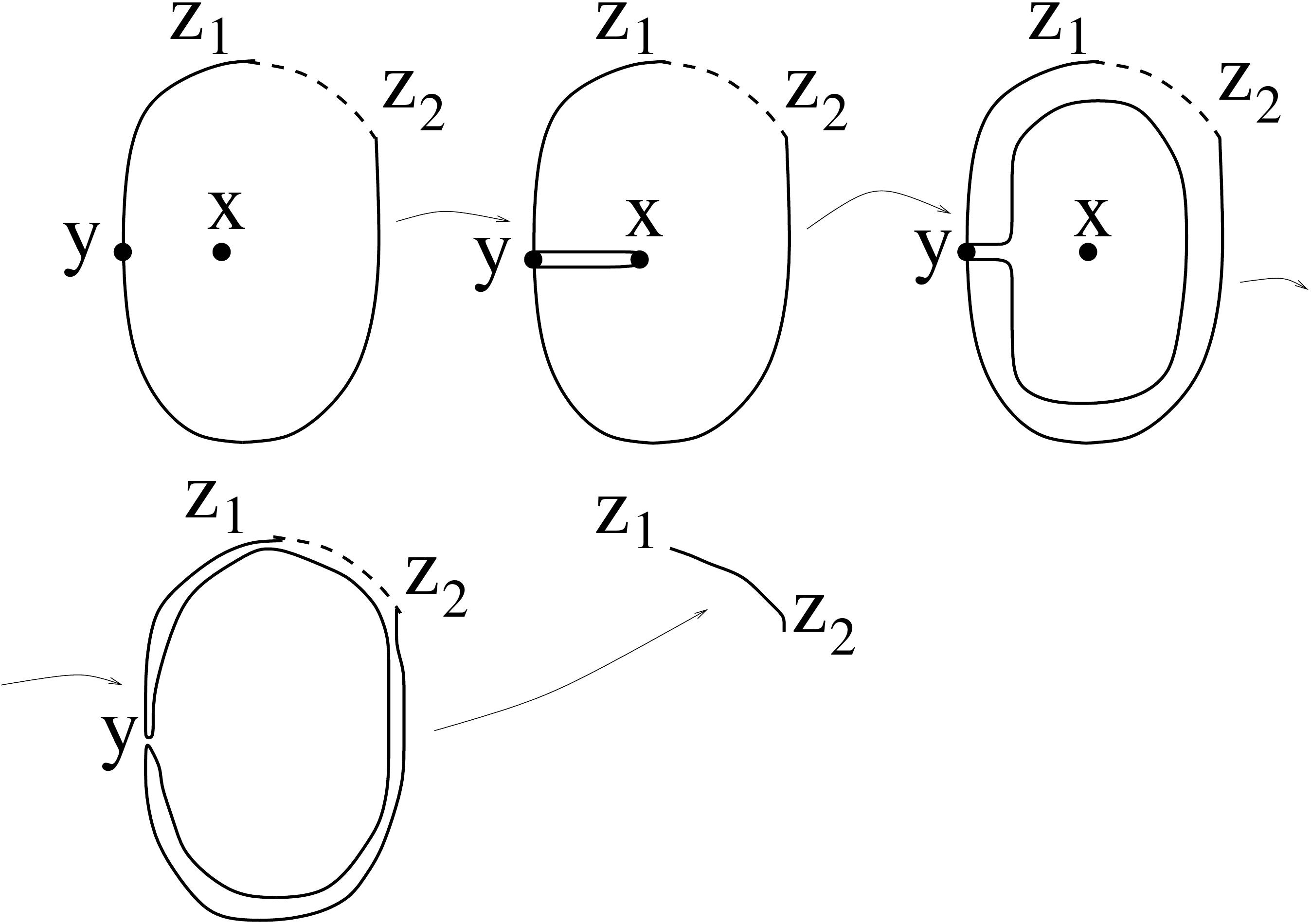} 
\caption{}
\label{LS12}
\end{figure}

\par
Now, one can start from the only digon, where $r>{d\over 2}$, if such a digon exists, create its boundary out of a point
by performing the Bikhoff curve-shortening in reverse, and then use the constructed path homotopies to move
sides of the digon through adjacent digons. We can continue in this way getting rid of all digons until only one of digons, $\Delta$,
will be left. We could arrange for $\Delta$ to be any digon other than
the digon we started from (which is the digon with $r\geq {d\over 2}$, if such a digon exists, and an arbitrary digon
otherwise). Then it remains only to construct a path homotopy contracting the boundary of $\Delta$ - for example, using
the Birkhoff curve-shortening process.
\par
{\bf 3.3. Homotopies that pass via simple curves.} It had been observed in [NR], that one can perturb this homotopy so that the resulting homotopy
would pass only through simple curves (as well as the points at the beginning and the end of the homotopy) (see Lemma 3.1 below for
a formal statement of such nature that will be
required for our purposes below). As the result, one obtains
a segment $I$ in the space of closed curves on $M$ such that its endpoints are trivial curves and all intermediate curves are simple. This segment
can be used to obtain via Morse theory (as above) one simple geodesic of controlled length ($\leq 5d$, if there are no geodesics of length $\leq 2d$
of index zero). 
\par
Alternatively, we could just start from a side of the initial digon, to connect it 
with the other side by a path homotopy
that passes through nonself-intersecting curves
of length $\leq 2r+3d+\epsilon$ for an arbitrary
small $\epsilon$ using a homotopy constructed in Lemma 3.1 below See
Fig. ~\ref{LS14}).

\begin{figure}[center] 
\includegraphics[scale=0.3]{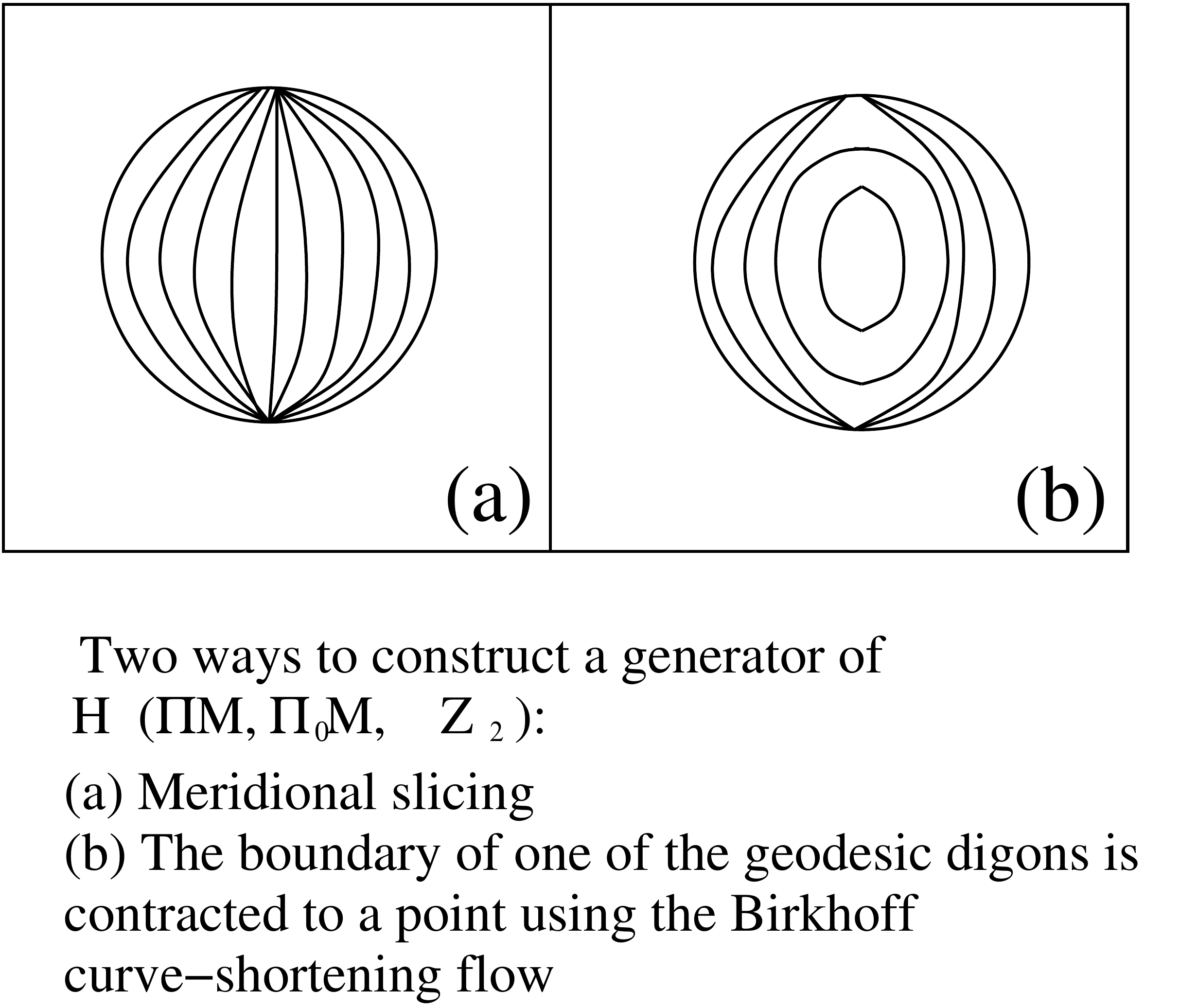} 
\caption{}
\label{LS14}
\end{figure}

Different curves in this path homotopy intersect each other only at their
common endpoints.
Then one
can consider a similar path homotopy through the adjacent digon, then through the next one, and so on, and continue in this manner, until
we return to the initial geodesic segment slicing out $M$ into nonself-intersecting segments of length $\leq 2r+3d+\epsilon\leq 5d+\epsilon$
passing through $p$ and $q$ and not intersecting at any other points.
Proposition 2.2 implies the existence of three distinct simple
periodic geodesics with lengths $\leq 6d$, $10d$ and $20d$, correspondingly.
However, the first bound can be improved to $5d$, as the resulting $1$-cycle in the space of non-parametrized
simple curves will be homologous to the $1$-cycle constructed in the previous paragraph. Therefore, the first of these simple periodic geodesics will coincide
with the 
periodic geodesic obtained in the previous paragraph, and its length will not exceed $5d$.
\par
{\bf 3.4. Path homotopies between arcs of convex curves}. Here we prove
the following lemma (some weaker versions
of which can be found in [Cr], [M] and [NR1]). We will present a self-contained
proof that incorporates some observations of Croke
([Cr]) and Maeda ([M]). We have already  used this lemma in section 3.3
to conclude that if $M$ has no simple closed geodesics of length
$\leq 2d$, then there are three distinct simple periodic geodesics on $M$
with lengths $\leq 5d,\ 10d$ and $20d$, correspondingly.

\begin{lemma} \label {Lemma3.1}
Let $D\subset M$ be a $2$-disc, and $\gamma_t,\ t\in[0,1],$ be a monotone homotopy between its boundary $\gamma_0$ 
and a point $x=\gamma_1$ in $D$. Assume that for each $t$ the length of $\gamma_t$ does not exceed $L$,
all curves $\gamma_t$ are convex with respect to the discs
contained in $D$,
and that
$y\in\gamma_0$ is a point such that $dist(x,y)=dist(x,\gamma_0)=r$. 
Then for each positive
$\epsilon$ there exists a monotone
homotopy $l_t, t\in [0,1],$ between $\gamma_0$ and the constant loop $y$ that
passes through simple loops based at $y$ and contained in $D$ such that
the lengths of each loop $l_t$ does not exceed $L+2r+\epsilon$.
Moreover, if $z_1, z_2$ are any two points on $\gamma_0$, then there exists
a path homotopy $A_t$ between two arcs $A_0$ and $A_1$ of $\gamma_0$ between
$z_1$ and $z_2$ such that for each $t$ the length of $A_t$ does not exceed
$\max\{$length$(A_1),\ $length$(A_0)\}+L+2r+\epsilon$, all arcs $A_t$ in this path homotopy
do not self-intersect, and
different arcs intersect each other only at their common
endpoints.
\end{lemma}
\par\noindent
\begin{proof} Let $\tau$ be a minimal geodesic from $y$ to $x$.
As $y$ is a closest to $x$ point of $\gamma_0$, $\tau$ is contained in $D$.
Now note that $\tau$
can intersect each curve $\gamma_t$ only at a point
or along an interval. 
Further, note that $\tau$ needs to intersect $\gamma_t$ at least once to
reach $x$. We are going to prove by contradiction that once $\tau$ enters the open disc $D_t$ containing $x$ bounded by $\gamma_t$
it will not be able to exit this disc. Indeed, if it exits $D_t$ it will
need to return to $\gamma_t$ later. This observation implies that 
for some $s<t$ $\tau$ touches $\gamma_s$ at some point $w$ but does not exit the closed 
disc bounded by $\gamma_s$. Now consider the minimizing geodesic in $M$
connecting two points $w_1, w_2\in \gamma_0$ that lie on the opposite sides
of $w$ and are very close to each other. The convexity of $\gamma_s$
implies that this geodesic must intersect $\tau$ at two different points,
which contradcts the minimality of $\tau$ (see Fig. ~\ref{LS6}).

\begin{figure}[center] 
\includegraphics[scale=0.3]{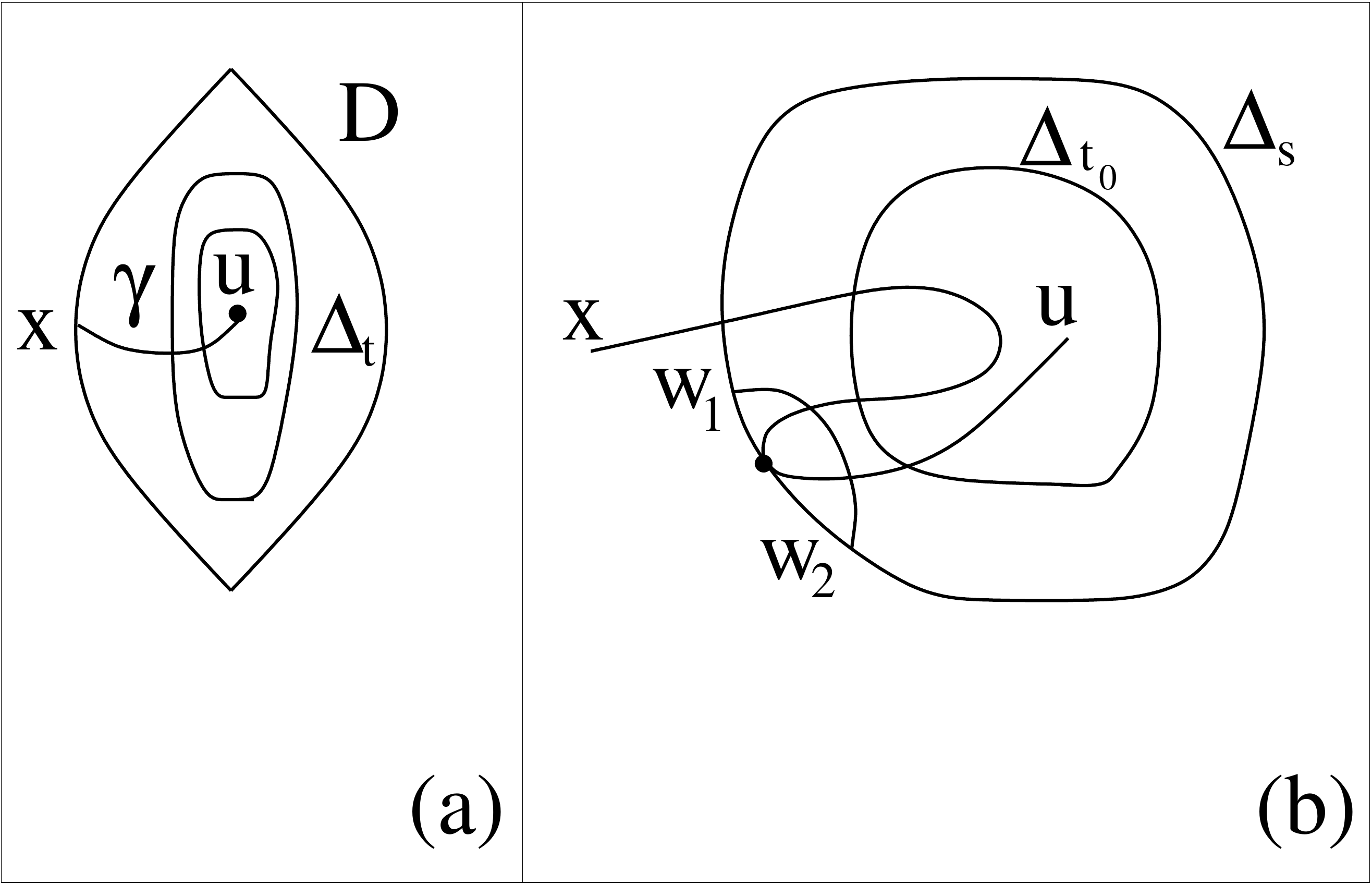} 
\caption{}
\label{LS6}
\end{figure}

In fact, we can continue
this argument and demonstrate that $\tau\bigcup\gamma_t$ is a point
and not a closed interval as it was done in section 3.2, but we do
not need this.

Now  we can construct
a homotopy $\tilde l_t$ between $\gamma_0$ and $\gamma_1$ as follows. For each $t$ we follow $\tau$ from
$\tau(0)$ to $\tau(t)$, then traverse $\gamma_{t^*}$ such that $\tau(t)\in
\gamma_{t^*}$ counterclockwise until we return back  to $\tau(t)$, then return
back along $\tau$ from $\tau(t)$ to $\tau(0)$. These curves will not be simple,
as we use an arc of $\tau$ between $\tau(0)$ and $\tau(t)$ twice, but an
obvious small perturbation makes these loops simple and pairwise intersecting only at $y$: We just move two segments of $\tau$ between $\tau(0)$ and $\tau(t)$
apart in different directions. For each $t$ the perturbation is very small,
but its magnitude increases with $t$. The point $\tau(t)$ is replaced by two very
close points on $\gamma_{t^*}$ on both sides of $\tau(t)$ that are used as 
the endpoints of the perturbed segments of $\tau$. The resulting
modfied homotopy will be the desired homotopy $l_t$.
\par
In order to prove the last assertion about the existence of the path homotopy
connecting two arcs between $z_1$ and $z_2$ we can proceed as follows (see
Fig. ~\ref{LS12}).
Denote the arc between $z_1$ and $z_2$ that contains $y$ by $A_0$.
Let $A_0^1$ and $A_0^2$ denote the arcs of $A_0$ from $z_1$ to $y$ and
from $y$ to $z_2$, respectively.
Start from $A_0$.
If we ran the homotopy of based loops in reverse direction, we can gradually 
grow loop $\gamma_0$ from point $y$. At the end of this homotopy we obtain 
the curve $A_0^1*\gamma_0*A_0^2$ connecting $z_1$ and $z_2$.
Now we can cancel
arcs between $z_1$ and $y$ as well as $z_2$ and $y$ (as each of these arcs
will be appearing twice with opposite orientations, and we can cancel these
arcs by means of obvious length non-increasing homotopies.)
\end{proof}

\section{Slicing of $M$ into short curves in the presence of one short simple periodic geodesic of index zero}

Now assume that there exists one non-trivial periodic geodesic $\gamma$ of length $\leq 2d$ of index $0$ which is simple and is the only simple periodic
geodesic of index $0$ and length $\leq 4d$.
We are going to construct a ``meridional slicing"
of $M$ into nonself-intersecting curves starting at $p$ and ending at $q$ and not intersecting at any other points of
length $\leq 10d+\epsilon$ for an
arbitrarily small $\epsilon$.
In this case Proposition 2.2 will imply the existence of two distinct simple periodic geodesics
of length $\leq 11d$ and $20d$ with positive
indices, so that they cannot coincide with the geodesic of index zero.
\par
Then we will obtain the first of these two geodesics of a positive index
using a homologous relative
$1$-cycle in the space of nonparametrized $1$-curves
that consists of curves of length $\leq 8d+\epsilon$. As the result, we will be able to conclude that the first of these two periodic geodesics
has length $\leq 8d$ (and not only $\leq 11d$). 
\par
Indeed, assume first that $\gamma$
does not arise during the Birkhoff-curve shortening process applied to the boundary of one of the geodesic digons considered in
the previous section obstructing its length non-increasing contraction to a point. Then we can proceed exactly as in
the previous section. So, without any loss of generality we can assume that $\gamma$ arises during the Birkhoff curve-shortening
process applied to the boundary of one of the digons and denote this digon
by $D$. (Recall
that the considered digons are formed by pairs of minimizing geodesics
connecting a pair of points $p,q\in M$ such that $dist(p,q)=d$. The angles
of each digon at both $p$ and $q$ do not exceed $\pi$.)
But $\gamma$ can obstruct the process of contracting the boundary
of only one of these digons. For the remaining digons we will be able to construct ``short" path homotopies from one side to the other.
After we will combine all these homotopies, we are going to have two options:
First, we can complete this homotopy by constructing
a path homotopy of one side of $\partial D$ to the other. Second, we can
contract the boundary of $D$ through $D$ to a point
as a free loop (see Fig. ~\ref{LS14}). The second option will lead to shorter closed curves but will yield only one ``short" simple geodesic
of index $1$. (In fact, Proposition 4.1 below will immediately imply that the boundary of $D$ can be contracted through simple closed curves of length $\leq 8d$.)
\par
So, assume that when one applies the Birkhoff curve-shortening process to $\partial D$ (which is assumed to be convex to $D$), the process
stops at a simple periodic geodesic $\gamma_0\subset D$. Of course, the length of $\gamma_0$ does not exceed $2d$.
Denote the piecewise geodesic convex closed curves arising during the Birkhoff curve-shortening process by $h_t$, where
$h_0=\partial D$ and $h_1=\gamma_0$.
\par
In the remainder of this section our goal will be to prove that
there exists a path homotopy between the two minimal
geodesics forming $\partial D$ that passes through nonself-intersecting
curves of length $\leq 10d$ that intersect each other only at their
common endpoints.
\par
\begin{figure}[center] 
\includegraphics[scale=0.3]{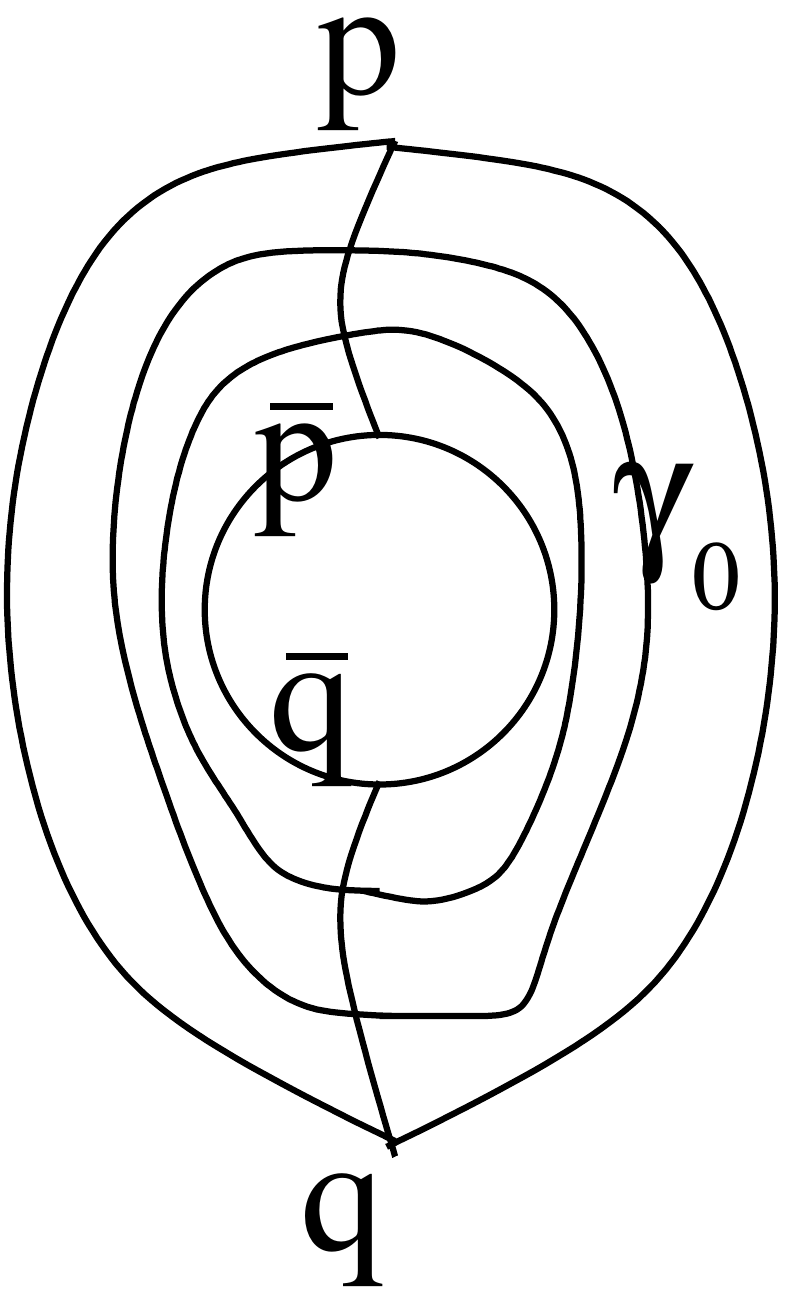} 
\caption{}
\label{LS8}
\end{figure}
\par
Let $\bar p$ and $\bar q$ denote points of $\gamma_0$ that are the closest to,
correspondingly, $p$ and $q$ (see Fig. ~\ref{LS8}).
Connect $p$ and $\bar p$ by a minimizing geodesic $\alpha$.
It intersects each convex piecewise geodesic curve  $h_t$ at a point. (It should enter
the domain bounded by $h_t$, but if it exists this domain, then it cannot return back
because of the convexity of all curves $h_s$ -see section 3.2 for the details
of this argument.) Also, it is easy to see that the point
of intersection of $\alpha$ and $h_t$ continuously depends on $t$. We can similarly connect $q$ and $\bar q$ by a minimizing
geodesic $\beta$. It will also intersect each $h_t$ at exactly one point that continuously depends on $t$.
\par
Now each of two sides of the digon $\partial D$ can be homotoped through curves that go along a segment of
$\alpha$ to its point of intersection with $h_t$, then go along one
of two arcs of $h_t$ to its point of intersection with $\beta$,
and then follow $\beta$ to $q$.
\par
In order to complete these two segments of a path homotopy to a desired
path homotopy
connecting both sides of $\partial D$ we need to connect two arcs
of $\gamma_0$ between $\bar p$ and $\bar q$ by a path homotopy 
passing through paths of length $\leq 8d+\epsilon$ (and to attach $\alpha $ and $\beta$ to the ends
of each of these paths) (see Fig. ~\ref{LS7}).

\begin{figure}[center] 
\includegraphics[scale=0.3]{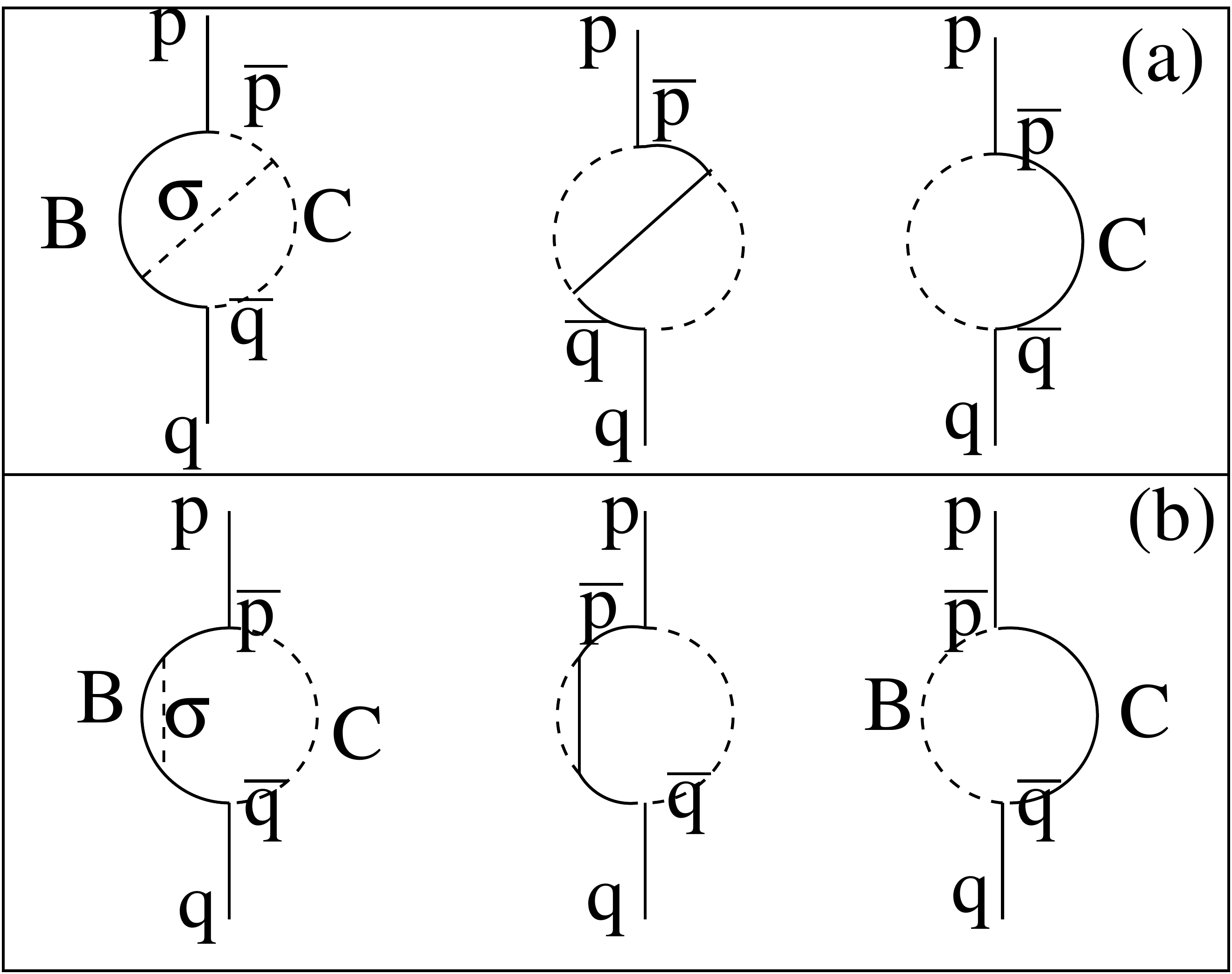} 
\caption{}
\label{LS7}
\end{figure}

\par
To achieve this last goal we first consider the case, when  there exists 
a pair of distinct points $r$ and $s$ on $\gamma_0$ that
can be connected by a minimizing geodesic segment $\sigma$ such that its
interior lies entirely
in the open subdisc of $D$ bounded by $\gamma_0$. Denote arcs of $\gamma_0$
connecting $r$ and $s$ by $B$ and $C$. Then digons $B\bigcup \sigma$ and
$\sigma\bigcup C$ are convex. These two digons can be contracted 
by means of a length nonincreasing homotopy using the Birkhoff curve-shortening
procedure, as our assumptions imply that  there are no ``short" non-trivial
periodic geodesics inside of these two digons. Therefore there
exists a path homotopy between $B$ and $C$
that passes through $\sigma$ and consist of paths of length $\leq 6d+\epsilon$ connecting
$r$ and $s$. (Here we are using Lemma 3.1.)
Now we need to consider cases when both points $\bar p$ and $\bar q$ are on the same arc, say, $B$, and another case when one of two points (say, $\bar q$) is on $B$
and another on $C$. In the first case we first attach longer and longer segments
of $B$ traversed in both directions at the beginning and and at the end
of the arc conecting $\bar p$ and $\bar q$ inside $B$. At the end we 
obtain the arc that goes from $\bar p$ to one of the ends of $B$, then follows
$B$ to its other end and then returns to $\bar q$. Now we homotope the 
middle part of the path, $B$, to $C$ via paths of length $\leq 6d+\epsilon$ keeping
its endpoints fixed. We end up
with the compementary arc of $\gamma_0$ connecting $\bar p$ and $\bar q$. The 
resulting homotopy
passes through paths of length not exceeding $8d+\epsilon$. In the second case 
we start from an arc connecting $\bar p$ and $\bar q$, then elongate
its subarc in $B$ by inserting longer and longer  complementary
subarcs of $B$ until the midle part of the curve will be $B$. Again,
we homotop it to $C$ keeping the endpoints fixed, and remove the unnecessary
arc of $C$ which will be traversed twice in different directions.
\par
Now we consider the case, when no such $r,s\in\gamma_0$ exist.
In this case for each point
$x\in\gamma_0$ any minimizing geodesic between $p$ and $x$ is in the
annulus formed by $\partial D$ and $\gamma_0$. (Of course, the same is true
for $q$ and $x$.) Indeed, this minimizing geodesic cannot cross two minimal
geodesics forming $\partial D$, and also cannot cross $\gamma_0$ because
of our assumption. Now note that in the argument above we do not need to assume
anymore that $\bar p$ is the closest to $p$ point on $\gamma_0$, and
$\bar q$ is the closest to $q$ point on $\gamma_0$. Instead we can
choose $\bar p$ and $\bar q$ to be arbitrary points on $\gamma_0$
such that minimal geodesics between $p$ and $\bar p$ 
as well as $q$ and $\bar q$ do not intersect.
But before explaining how to choose $\bar p$ and $\bar q$
we are going to state the following proposition and review
its proof given in [NR]:


\par\noindent
\begin{Pro} \label {Proposition 4.1.} Let $\gamma_0$ be a simple closed curve of length $L$ in
a Riemannian surface of diameter $d$ that bounds a disc $D_0$. Denote
$\max_{x\in D_0}dist(x,\gamma_0)$ by $r$. Curve $\gamma_0$ might be a periodic
geodesic of index $0$, but assume that there are no other non-trivial periodic geodesics of index $0$ of length $\leq L+2r$ in $D_0$.
Then  for each $\epsilon>0$ $\gamma_0$ can be contracted to a point via simple closed curves of length $\leq \max\{{L\over 2},2r\}+ L+4r+\epsilon$. Moreover,
these curves are pairwise non-intersecting and are contained in $D_0$. 
\end{Pro}
\par\noindent
\begin{proof} Let $z$ be a point in $D_0$ at the maximal distance
from $\gamma_0=\partial D$. Then there exist points $x_1,x_2,\ldots x_k\in\gamma_0$ such that
for each $i$ $dist(z, x_i)=dist (x,\gamma_0)$, $k\geq 2$,
and tangent vectors at $z$ to minimizing geodesics from $z$ to $x_i$
subdivide the tangent space $T_zM$ into angles $\leq\pi$ (see Fig.  ~\ref{LS1} ).

\begin{figure}[center] 
\includegraphics[scale=0.3]{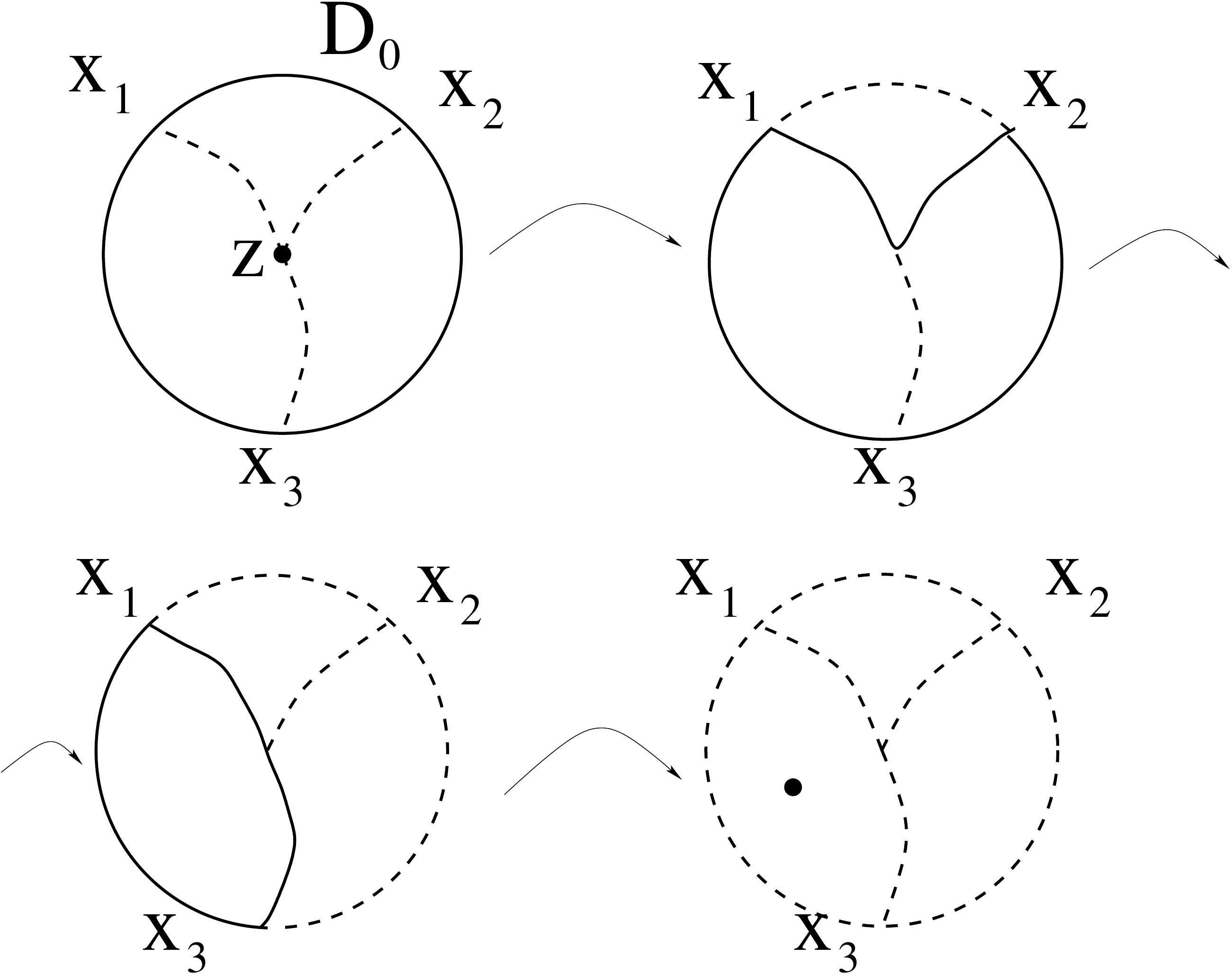} 
\caption{}
\label{LS1}
\end{figure}

(Otherwise, one
can move $z$ further away from $\gamma_0$ as in the proof of Berger's Lemma.)
As the geodesics $zx_i$ are perpendicular to $\gamma_0$ at $x_i$, they
divide $D_0$ into convex triangles $T_i=zx_ix_{(i\mod k) +1}$.
Birkhoff's curve-shortening process contracts the boundary of each triangle
$T_i$ to a point inside this triangle; the resulting homotopy $H_i$
is monotone and length non-increasing. Lemma 3.1 implies
that each homotopy $H_i$ can be easily
converted into a path homotopy connecting one side (or two sides) of $T_i$
with the curve formed by the two other sides (respectively, the other side)
such that this path homotopy passes through
$T_i$ and ``eliminates" $T_i$. Lemma 3.1 also provides an upper bound
for the lengths of curves in this path homotopy.
We can assume that the triangle corresponding
to the longest arc $x_ix_{(i\mod k)+1}$ of $\gamma_0$ will be the
last one being contracted. (This helps to make an upper bound for the lengths
of curves during the homotopy somewhat smaller.)
An easy calculation proves the claimed upper bound for the lengths of curves
in the homotopy contracting $\gamma_0$ to a point.
\end{proof}

Consider the triangles $T_i$ defined in the previous proof. Among them
choose a triangle with the maximal value of
$\max_{y\in T_i}$ dist$(y,\partial T_i)$. Denote this triangle
by $T_m$, and its vertices
on $\gamma$ by $x_m$ and $x_{m+1}$. We choose the pair of points $x_m, x_{m+1}$ as the pair
$\bar p, \bar q$.
%

More precisely, if $dist(p, x_m)+dist(q, x_{m+1})\leq
dist(p, x_{m+1})+dist(q, x_m)$, then $\bar p=x_m,\ \bar q=x_{m+1}$, and,
otherwise, $\bar p=x_{m+1},\ \bar q=x_m$. This ensures that minimal geodesics between $p$ and $\bar p$ and between $q$ and $\bar q$  do not intersect.

Note that there is at most one triangle $T_i$ such that 
$\max_{y\in T_j}$dist$(y,\partial T_j)> {d\over 2}$. Indeed, otherwise
there would be two triangles $T_{j_1}, T_{j_2}$ and points $z_1\in T_{j_1}$,
$z_2\in T_{j_2}$ such that dist $(z_1,z_2)>d$, which is impossible.

\begin{figure}[center] 
\includegraphics[scale=0.3]{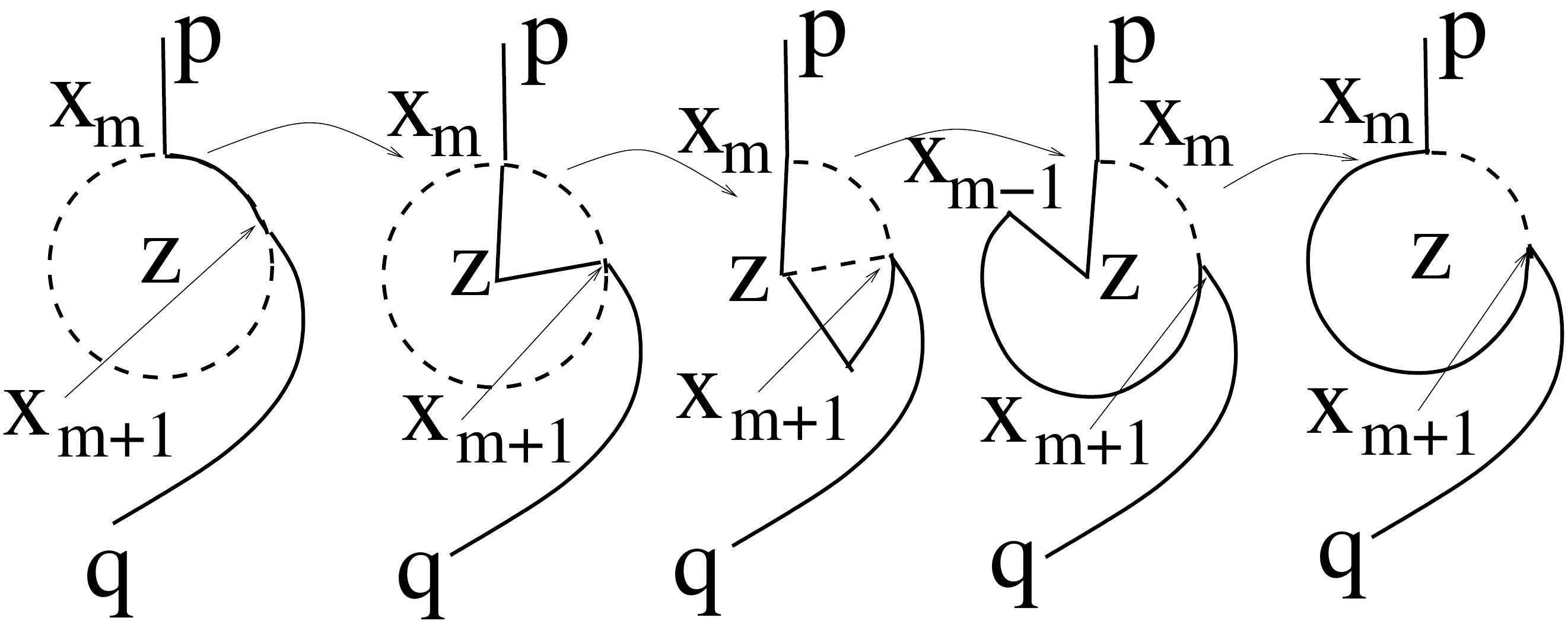} 
\caption{}
\label{LS13}
\end{figure}

Now it remains to find a path homotopy between the side $(x_mx_{m+1})$ of $T_m$ and its complementary arc in $\gamma$ that passes through  simple curves of length $\leq 8d+\epsilon$
contained in $D_0$ (see Fig. ~\ref{LS13}). This path homotopy will be constructed in stages.
On the first stage we construct a path homotopy in $T_m$ between $x_mx_{m+1}$ and the broken geodesic formed by two
other sides $x_mzx_{m+1}$ of the triangle $T_m$. We apply Lemma 3.1 to construct
this path homotopy, and this lemma implies that the length of the
paths during the homotopy will not exceed $8d+\epsilon$. (Indeed, using the notations
of Lemma 3.1, $L\leq 4d$, $2r\leq 2d$ and
$\max\{$length$(A_1),$ length$(A_2)\}\leq 2d$.)
On each of the next stages we take either a triangle $T_j$ 
  adjacent to one of the triangles $T_i$ that were already covered by the image
of the constructed path homotopy that is the furthest in the clockwise direction (typically, $j=i+1$), and construct a path homotopy in $T_j$ between
the common side 
of $T_i$ and $T_{j}$ 
and two other sides
of $T_j$, 
or do the same for the triangle adjacent to the
``processed" triangle that is the furthest in the counterclockwise direction. One of these two other sides of $T_j$ will be a minimal geodesic from $z$ to a new vertex of $T_j$,
the other will be an arc of $\gamma$ that will
be added to already constructed part of $\gamma\setminus (x_mx_{m+1})$.
The common side $zx_i$ of $T_i$ and $T_j$ as well as an arc of $\gamma$
constructed on previous stages remain attached at the beginning of all paths
obtained during this path homotopy. Note that we have a choice of
which of the
remaining triangles (other than the first one) to eliminate at the
last step of the procedure. We choose the last traingle to be the triangle
such that its side on $\gamma$ is the longest among the remaining triangles.
This automatically
means that the lengths of the side on $\gamma$ of each of the intermediate 
triangles will not exceed $d$. In this way we ensure that when we apply
Lemma 3.1 on each of the intermediate steps
$\max\{$length$(A_1),$ length$(A_2)\}\leq 2d$.
On the last step we eliminate the last triangle that we denote $T_{s}$ by constructing the path 
homotopy between two sides of $T_{s}$ adjacent to $z$ and the third side
which is an arc of $\gamma$. The already constructed arc of $\gamma$
complementary to the side of $T_s$ on $\gamma$ is attached to all these paths.
Again, an application of Lemma 3.1 implies that all these paths have the length
not exceeding $8d+\epsilon$.
(Indeed, as before, $\max\{$ length$(A_1),\ $length$(A_2)\}$
does not exceed $2d$. Also note
that on each step the total length of all arcs of $\gamma$ that already
were constructed as well as the new one that enters $d$ does not exceed $2d$.
On all of the steps starting from the second one $2r\leq d$, but on
each of the intermediate steps (between the first step and the last step)
we also keep $zx_i$ fixed during the homotopy, and the length of this segment
adds at most $d$ to our upper bound for the lengths of curves during
the constructed homotopy.
Finally, the lengths of two geodesics
between $z$ and two other vertices of $T_{s}$ also enter $L$ on the last step.
Their total length also does not exceed $2d$. Summarizing all these
observations we see that the upper estimate provided by Lemma 3.1 on each of the steps of our construction is $8d+\epsilon$, as desired.

\section{Slicing of $M$ into short curves in the presence of more than one short simple periodic geodesic of index zero}

Now we are going to consider the case, when there exists a simple 
periodic geodesic of length $\leq 2d$ and index $0$ as well as another
simple
periodic geodesic of length $\leq 4d$ and index $0$.
If there are three simple periodic geodesics of length $\leq 6d$, then we are
done.  Therefore, without any loss of generality we can assume that there
is no third simple
periodic geodesic of length $\leq 6d$ of index $0$.
\par
Our goal will be to construct a relative $1$-cycle of $(\Pi M,\Pi_0 M)$
representing the non-trivial homology class with $\mathbb{Z}_2$ coefficients.
This relative $1$-cycle will be
represented by a map of an interval so that the endpoints
of the interval are mapped into constant curves, and all intermediate
points are mapped into simple closed curves of length $\leq 20d$. As we
saw in section 2, this will be sufficient to complete the proof of the theorem.
\par
As in the previous section we will consider the digons formed
by minimal geodesics $p$ and $q$, where $d(p,q)=d$  and all angles
in all digons do not exceed $\pi$. We will be trying to connect
the sides of each digon by a path homotopy passing
through curves of length $\leq 8d+\epsilon$. This will be obviously sufficient for our purposes.
If the boundary of each digon can be contracted by means of the Birkhoff
curve-shortening process to a point, then Lemma 3.1 implies that
this is, indeed, possible, and we are done.
But, in principle, an application of the Birkhoff curve-shortening process
can end at a simple closed geodesic of length $\leq 2d$. This can happen
with one or two digons.
\par
Case 1. The application of the Birkhoff process to two distinct
digons results in simple closed geodesics. In this case we can be sure that
1) the lengths of these geodesics do not exceed $2d$; 2) These geodesics
are simple and distinct; 3) They have index $0$. Now our assumption
implies that there are no other simple geodesics of index $0$ and
length $\leq 6d$.
This means that each of these two periodic
geodesics can be contracted to a point trhough simple pairwise non-intersecting
closed curves of length $\leq 8d$ as in Lemma 4.1. Moreover, the images 
of these two homotopies are contained in the discs bounded by corresponding
geodesics that are contained in the interiors of the corresponding digons.
Now we can use one of these homotopies to go from a point to
one of these two geodesics. Then we use the Birkhoff curve-shortening homotopy
in reverse to go from the geodesic to the boundary of the corresponding digon.
Then we continue our homotopy by moving one of two minimizing
geodesics though neighboring digons using the fact that their boundaries
can be contracted by means of the Bikhoff curve-shortening process and
Lemma 3.1. We arrange this stage so that the last remaining digon will
be the second digon that cannot be contracted by means of the Birkhoff
curve-shortening proces to a point. Once we obtain this digon we apply
the Birkhoff curve-shortening process that ends at the second simple
periodic geodesic. Finally, we contract this geodesic inside the disc
contained in the second digon via simple closed curves of length
$\leq 8d+\epsilon$ as in Lemma 4.1. As the result, we obtain a desired relative 
$1$-cycle in $(\Pi M,\Pi_0 M)$, where the lengths of all curves do not
exceed $8d+\epsilon$.
\par
Case 2. The application of the Birkhoff curve-shortening process
to all digons but one contracts them to points. The application
of the Birkhoff curve shortening process to the last digon
ends at a simple periodic geodesic $\gamma_1$. The length
of $\gamma_1$ does not exceed $2d$. Denote the disc bounded by
$\gamma_1$ inside the digon by $D$, and its complement by $E$. One can contract $\gamma_1$
through ``short" simple pairwise non-intersecting curves to a point
inside $E$ as follows: First, connect $\gamma_1$ and the boundary
of the corresponding digon by means of the Birkhoff curve-shortening process,
where time goes in the opposite direction. Then keep one of two 
sides of this digon fixed and move the other side through all digons one
by one using Lemma 3.1. When there will be only one last digon left, apply
the Birkhoff curve-shortening process to contract it. The length
of the closed curves during this homotopy does not exceed $6d+\epsilon$.
\par
Now we would like to contract $\gamma_1$ to a point inside $D$ via simple
pairwise non-intersecting closed curves of length $\leq 10d+\epsilon$. The combination
of these two homotopies will yield the desired relative $1$-cycle.
\par
As in the proof of Lemma 4.1, choose a point $x\in D$ that maximizes
$dist (x,\gamma_1)$. Minimal geodesic segments from $x$ to $\gamma_1$
subdivide $D$ into a finite number of convex triangles. The application
of the Birkhoff curve-shortening process to all of these triangles but
possibly one contract them to points. As in the proof of Lemma 4.1 we can
move arcs of $\gamma_1$ through all these triangles but possibly one
eliminating these triangles in the process. The lengths of all curves during these
homotopies will not exceed $9d+\epsilon$. At the end of this stage we obtain
the boundary of the last triangle. If the application of
the Birkhoff curve-shortening process to this curve ends at a point,
then we are done. Otherwise, we end at the second simple periodic
geodesic of index $0$. Denote this geodesic by $\gamma_2$. The length of $\gamma_2$ does
not exceed $4d$. The open disc bounded by $\gamma_2$ inside the last triangle
cannot contain a simple periodic geodesic of index $0$ and length $\leq 6d$.
Therefore, another application of Lemma 4.1 implies that we can contract
$\gamma_2$ to a point inside this disc through simple pairwise
non-intersecting curves of length $\leq 10d+\epsilon$, and we are done.

\section{Concluding remarks.}

Observe that, in fact, we proved the following theorem:

\begin{theorem} \label {Theorem5.1}
Let $M$ be a Riemannian $2$-sphere of diameter $d$. There exist three
distinct simple periodic geodesics on $M$ that satisfy exactly
one of the following
four sets of (mutually exclusive) conditions:
\par\noindent
(i) All three geodesics have index $0$. Their lengths do no exceed $2d$, $4d$
and $6d$, correspondingly.
\par\noindent
(ii) Two of these geodesics have index $0$. Their lengths do not exceed
$2d$ and $4d$. The third geodesic has a positive index and length $\leq 10d$.
\par\noindent
(iii) One of these geodesics has index $0$ and length $\leq 2d$. Two other
geodesics have positive indices and lengths $\leq 8d$ and $\leq 20d$,
correspondingly.
\par\noindent
(iv) All three geodesics have positive indices. Their lengths do not exceed
$5d$, $10d$ and $20d$, correspondingly.
\end{theorem}

Our last remark pertains to the situation, when one  does not
have a meridional slicing of $M$ into nonself-intersecting curves of
length $\leq Cd$ for some large $C$. (Recall, that the existence of such a 
slicing would imply an upper bound $const\ d$ for lengths of three
simple periodic geodesics of positive index that appear in the original
Lyusternik-Shnirelman proof.) Note, that
iterating the same idea as in the last section,
we will be able to conclude the existence of $k$
simple pairiwise non-intersecting
geodesics with length $\leq 2d$, $\leq 4d$, $\ldots$, $\leq 2kd$, where
$k\sim {C\over 4}$.










\begin{tabbing}
\hspace*{7.5cm}\=\kill
Yevgeny Liokumovich                 \>Alexander Nabutovsky\\             
Department of Mathematics           \>Department of Mathematics\\ 
University of Toronto               \>University of Toronto\\
Toronto, Ontario M5S2E4             \>Toronto, Ontario M5S 2E4\\
Canada                              \>Canada\\
email: liokumovich@math.toronto.edu \>alex@math.toronto.edu\\
\end{tabbing}

\par\noindent
Regina Rotman
\par\noindent
Department of Mathematics
\par\noindent
University of Toronto
\par\noindent
Toronto, Ontario M5S2E4
\par\noindent
Canada
\par\noindent
rina@math.toronto.edu
\par\noindent

\end{document}

%% file: figsa.tex
\renewcommand{\marginpar}[1]{}
\catcode`\@=12

\def\Empty{}
\newcommand\oplabel[1]{
  \def\OpArg{#1} \ifx \OpArg\Empty {} \else
  	\label{#1}
  \fi}
		
\newcommand{\comm}[1]{}